\documentclass[12pt]{article}
\usepackage[margin=1in]{geometry}
\usepackage[utf8]{inputenc}
\usepackage[T1]{fontenc}
\usepackage{amssymb}
\usepackage{amsmath}
\usepackage{amsthm}
\usepackage{amsfonts}
\usepackage{esint}
\usepackage{dsfont}
\usepackage{wasysym}
\usepackage{bbm}
\usepackage[all]{xy}
\usepackage{hyperref}
\usepackage{caption}
\usepackage{enumitem}
\usepackage{graphicx}
\usepackage{tikz-cd}
\usepackage{multirow}
\usepackage{array}
\usepackage{parskip}
\newcolumntype{P}[1]{>{\centering\arraybackslash}p{#1}}\newcolumntype{M}[1]{>{\centering\arraybackslash}m{#1}}

\usetikzlibrary{decorations.markings,decorations.pathreplacing,positioning,cd}

\newtheorem{theorem}{Theorem}[section]
\newtheorem{corollary}[theorem]{Corollary}
\newtheorem{lemma}[theorem]{Lemma}
\newtheorem{proposition}[theorem]{Proposition}
\newtheorem{claim}{Claim}

\theoremstyle{definition}
\newtheorem{definition}[theorem]{Definition}
\newtheorem{assumption}[theorem]{Assumption}

\theoremstyle{remark}
\newtheorem*{remark}{Remark}

\newcommand{\R}{\mathbb{R}}
\newcommand{\Z}{\mathbb{Z}}

\newcommand{\Prob}{\mathbb{P}}
\newcommand{\Proba}[1]{ \Prob\left( #1 \right)}

\newcommand{\norm}[2]{\left\Vert #1 \right\Vert_{#2}}

\newcommand{\scal}[2]{\left<#1,#2\right>}

\newcommand{\excur}{\mathcal{E}_\ell}
\newcommand{\connects}{\overset{\excur(f)}{\longleftrightarrow}}
\newcommand{\chem}{\text{chem}}
\newcommand{\diam}{\text{diam}}
\newcommand{\dchem}{d_{\chem}^{\excur(f)}}

\providecommand{\keywords}[1]
{
  \textbf{\textit{Keywords---}} #1
}
\providecommand{\subjclass}[1]
{
  \textbf{\textit{Subject classification---}} #1
}

\title{Chemical distance for smooth Gaussian fields in higher dimension.}
\author{David Vernotte}	

\begin{document}

\maketitle

\begin{abstract}
Gaussian percolation can be seen as the generalization of standard Bernoulli percolation on $\mathbb{Z}^d$. Instead of a random discrete configuration on a lattice, we consider a continuous Gaussian field $f$ and we study the topological and geometric properties of the random excursion set $\mathcal{E}_\ell(f) := \{x\in \R^d\ |\ f(x)\geq -\ell\}$ where $\ell\in \mathbb{R}$ is called a level. It is known that for a wide variety of fields $f$, there exists a phase transition at some critical level $\ell_c$. When $\ell> \ell_c$, the excursion set $\mathcal{E}_\ell(f)$ presents a unique unbounded component while if $\ell<\ell_c$ there are only bounded components in $\excur(f)$.
In the supercritical regime, $\ell>\ell_c$, we study the geometry of the unbounded cluster. Inspired by the work of Peter Antal and Agoston Pisztora for the Bernoulli model \cite{Antal}, we introduce the chemical distance between two points $x$ and $y$ as the Euclidean length of the shortest path connecting these points and staying in $\excur(f)$. In this paper, we show that when $\ell>-\ell_c$ then with high probability, the chemical distance between two points has a behavior close to the Euclidean distance between those two points.
\end{abstract}


\tableofcontents

\section{General introduction and notations}

This paper studies some geometric properties of excursion sets of a large class of random Gaussian fields in the Euclidean space of dimension $d\geq 2$. The continuous model of random Gaussian fields and the percolation model associated to it have gathered a lot of interest over the past few years (see for instance \cite{BG16},\cite{quasi_independance},\cite{Threshold},\cite{severo2022uniqueness},\cite{lcd3} for an incomplete list of references). We present the model as a continuous analogue of the classical discrete percolation (Bernoulli percolation). In the discrete case, say site percolation on $\mathbb{Z}^d$ for convenience, we consider a configuration $(\omega_i)_{i\in \mathbb{Z}^d}\in [0,1]^{\Z^d}$, where each $\omega_i$ is a uniform random variable over $[0,1]$ and all $\omega_i$ are mutually independent. Given a parameter $p\in [0,1]$ we are interested in the random set
\begin{equation}
    \mathcal{G}_p(\omega) := \{i\in \Z^d\ |\ \omega_i\leq p\}.
\end{equation}
The set $\mathcal{G}_p$ is a random set and its properties are well known. In particular, it is known that there is a phase transition at some $p_c(d)\in ]0,1[$. When $p<p_c(d)$ all clusters of $\mathcal{G}_p(\omega)$ are finite and therefore bounded (almost surely) while when $p>p_c$, $\mathcal{G}_p(\omega)$ presents a unique infinite cluster (meaning a cluster which contains an infinite number of vertices). Note that the configuration $\omega$ can be seen as a random function
$$\omega : \Z^d \to [0,1].$$ The continuous model replaces $\omega$ with a random continuous function
$$f : \R^d \to \R,$$
and for a level $\ell \in \mathbb{R}$ the analogue of $\mathcal{G}_p(\omega)$ is
\begin{equation}
\label{eq:excurf}
    \excur(f) := \{x\in \R^d\ |\ f(x)\geq -\ell\}.
\end{equation}
In the following, we properly define the model. We first describe how the random function $f$ is defined. We then give the main theorems related to the percolation model. Finally we introduce the chemical distance and state our main theorem.

In the rest of the paper, $f$ will denote a centered, stationary, continuous, Gaussian field on $\R^d$. That is, $f$ is a random collection $(f(x))_{x\in \R^d}\in \R^{\R^d}$ such that
\begin{itemize}
    \item for all $x_1,\dots,x_n\in \R^d$ the vector $(f(x_1),\dots,f(x_n))$ is a random Gaussian vector which is centered (Gaussian field);
    \item for all $x\in \R^d$, $f(x+\cdot)$ and $f(\cdot)$ have the same law (stationarity). (here and in the following $f(\cdot)$ stands for $y\mapsto f(y)$);
    \item almost surely $f : \R^d \to \R$ is continuous.
\end{itemize}
As an example, let us present the construction of such a field, the so-called Bargmann-Fock field. For this construction we work with a probability space where is defined $(a_{i_1,\dots,i_d})_{i_1,\dots,i_d\geq 0}$ a countable collection of mutually independent standard Gaussian random variables. For $x=(x_1,\dots,x_d)\in \R^d$ we set:
\begin{equation}
    f(x) := e^{-\frac{1}{2}\norm{x}{}^2}\sum_{i_1,\dots i_d\geq 0}a_{i_1,\dots,i_d}\frac{x_1^{i_1}\dots x_d^{i_d}}{\sqrt{i_1!\dots i_d!}}.
\end{equation}
Here and in the following, $\norm{\cdot}{}$ denotes the usual Euclidean norm.
A simple computation shows that almost surely, $f$ is well defined on $\R^d$ and by dominated convergence it is almost surely continuous (in fact $\mathcal{C}^\infty$). Moreover $f$ is clearly a Gaussian field which is centered and $\mathbb{E}[f(x)f(y)]=e^{-\frac{1}{2}\norm{x-y}{}^2}$ depends only on $x-y$ implying that $f$ is stationary. Moreover, since $\mathbb{E}[f(x)f(y)]$ only depends on $\norm{x-y}{}$, we see that the law of the field $f$ is actually isotropic. It is good to have in mind that our study applies for this specific field. Now that we introduced one example, we present a more general way to construct a variety of centered, stationary, continuous Gaussian fields. This is done via the white noise representation.
\begin{definition}
    A white noise $W$ on $\R^d$ is a centered Gaussian field indexed by functions of $L^2(\R^d)$ such that for any $\varphi_1,\varphi_2\in L^2(\R^d)$ we have
    $$\mathbb{E}[W(\varphi_1)W(\varphi_2)]= \int_{\R^d}\varphi_1(x)\varphi_2(x)dx.$$
\end{definition}
In the following, we will work on a probability space where a white noise $W$ is defined. To see an explicit construction of a white noise one may refer to \cite{janson}.
Consider a function $q : \R^d \to \R$ satisfying the following Assumption \ref{a:a1} for some $\beta>d$.
\begin{assumption}
\label{a:a1}
Let $\beta>d$. We say that $q$ satisfy Assumption \ref{a:a1} for $\beta$ if
\begin{itemize}
    \item $q$ is $\mathcal{C}^{10}$ and $\partial^\alpha q\in L^2(\R^d)$ for all $\alpha\in \mathbb{N}^d$ such that $\norm{\alpha}{1}\leq 10$ (regularity).
    \item $\max(|q(x)|,\norm{\nabla q(x)}{})=O(\norm{x}{}^{-\beta})$ (decay of correlation).
    \item $q$ is isotropic ($q(x)$ only depends on $\norm{x}{}$) (symmetries).
    \item $q\geq 0$ (strong positivity).
    \item $q$ is not identically equal to the zero function (non trivial).
\end{itemize}
\end{assumption}
Given a function $q$ which satisfies Assumption \ref{a:a1} for some $\beta>d$ one can define
$$f:= q \ast W,$$
where $\ast$ denotes convolution. That is, for all $x\in \R^d$ we set $f(x):=W(q(x-\cdot)),$ where we recall that $q(x-\cdot)$ is the function $y \mapsto q(x-y)$ which is in fact in $L^2(\R^d)$. It follows that $f$ is a centered Gaussian field. Its covariance kernel can be computed: $\mathbb{E}[f(x)f(y)]=(q \ast q)(x-y)$. This shows that the field $f$ is stationary. Moreover, by dominated convergence we can instead consider a continuous (in fact $\mathcal{C}^4$) modification of $f$, which we do in the rest of the paper. The white noise representation of the field is useful to define approximations of the field which have good properties (like finite range dependence, see Definitions \ref{def:fR} and  \ref{def:discretization}).

 We now briefly comment on Assumption \ref{a:a1}. The fact that we ask $q\in \mathcal{C}^{10}$ and the fact that the partial derivatives of $q$ are in $L^2$ is to guarantee that the field $f=q\ast W$ is not only continuous but at least $\mathcal{C}^4$. The assumption on the decay of $q(x)$ and $\nabla q(x)$ will give us certain decay of correlations of the random field. We require $q$ to be isotropic so that the law of the field presents useful symmetries (like rotation invariance for instance). The assumption $q\geq 0$ is often referred to as the \text{strong positivity} hypothesis, it implies in particular that the field $f$ has the FKG inequality. We comment that we do not believe these assumptions to be optimal, it might be possible to follow our argument with weaker assumptions. Note however that these assumptions are pretty classical in the context of Gaussian percolation (see for instance \cite{BG16}, \cite{severo2022uniqueness}, \cite{lcd3}).
\begin{remark}
If we set $q(x) = \left(\frac{2}{\pi}\right)^{d/4}e^{-\norm{x}{}^2}$ then $f=q \ast W$ obtained has the law of the Bargmann-Fock field previously introduced. Thus, note that the Bargmann-Fock field satisfies Assumption \ref{a:a1} for all $\beta > d$.
\end{remark}
\begin{remark}
 One may wonder what assumptions are needed for a centered, stationary, continuous Gaussian field to admit a white noise decomposition. Given such a field $f$, denote by $\kappa$ its covariance kernel, that is $\forall x\in \R^d, \ \kappa(x)=\mathbb{E}[f(x)f(0)].$ Since $f$ is continuous, Bochner's theorem allows to define the \textit{spectral measure} $\mu$ on $\R^d$ as the inverse Fourier transform of $\kappa$, meaning that
 $$\forall x\in \R^d,\ \kappa(x) = \int_{\R^d}e^{2i\pi \scal{x}{s}}\mu(ds).$$
 A sufficient condition for the white noise decomposition to hold is that $\mu$ has a density (called the \textit{spectral density} and denoted by $\rho^2$) with respect to the Lebesgue measure. This condition implies in particular that the law of $f$ is ergodic with respect to the flow of translations on $\R^d$, and this condition is implied by the fact that $x\mapsto \kappa(x)$ is integrable (this holds as soon as we have fast enough decay for $\kappa(x)$).
 We now argue that the existence of the spectral density implies the existence of a white noise decomposition of the field $f$. In fact, since $\int_{\R^d}\rho(x)^2dx=\kappa(0)<\infty$ the function $\rho$ belongs to $L^2(\R^d)$. We can therefore define $q$ as the Fourier transform of $\rho$ (it is again a function in $L^2(\R^d)$). Then the field $g=q \ast W$ has covariance kernel given by $q \ast q$, which is the Fourier transform of $\rho^2$, the spectral density. Since $\kappa$ was the Fourier transform of the spectral measure, we conclude that $f$ and $g$ have the same covariance kernel and the same law.
\end{remark}

We now introduce the percolation model associated to the continuous setting. Given a real parameter $\ell \in \R$ which we call a \textit{level}, we define the \textit{excursion set at level $\ell$} as
\begin{equation}
\label{eq:excurf}
    \excur(f) := \{x\in \R^d\ |\ f(x)\geq -\ell\}.
\end{equation}
We remark that since $f$ is random, the set $\excur(f)$ is also a random set. It appears that for a fixed realization of $f$, the collection $(\excur(f))_{\ell\in \R}$ is a non-decreasing sequence for set-inclusion as $\ell$ increases. Thus, the probability that $\excur(f)$ contains an unbounded component is non-decreasing in $\ell$ and we define the critical probability for percolation as
\begin{equation}
    \label{eq:def_lc}
    \ell_c := \sup\left\{\ell\in \mathbb{R}\ |\ \Proba{\excur(f)\text{ has no unbounded connected component}}=1\right\}.
\end{equation}
We remark that while we ask for an infinite cluster in the context of Bernoulli percolation, this would not make much sense in the case of continuous Gaussian percolation. In fact, for fixed $\ell\in \R$ then, almost surely, all connected components of $\mathcal{E}_\ell(f)$ will have a non empty interior and will contain an infinite number of points. Hence it is more interesting to know whether the connected components of $\mathcal{E}_\ell(f)$ are bounded or unbounded.
It turns out that when $d\geq 2$, then for a large class of Gaussian field we have $\ell_c\in ]-\infty,\infty[$, that is, there exists a non trivial phase transition. It is important to have in mind that the critical level $\ell_c$ is to be compared with the critical probability $p_c(d)$ in the Bernoulli setting on $\Z^d$.
For our continuous model, a lot is known about the case $d=2$. In particular the following theorem can be thought of as an analogue of the Theorem of Kesten \cite{Kesten} concerning the value of the critical probability of Bernoulli bond percolation on $\Z^2$.
\begin{theorem}[\cite{HA_critical}, \cite{Threshold}]
\label{thm:dim2}
In the case $d=2$, if $q$ satisfies Assumption \ref{a:a1} for some $\beta>2$ then we have $\ell_c=0$.
\end{theorem}
In the case $d\geq 3$ we also have the following result.
\begin{theorem}[\cite{lcd3}]
\label{thm:dim3}
In the case $d\geq 3$, if $q$ satisfies Assumption \ref{a:a1} for some $\beta>d$ then we have
$\ell_c<0$.
\end{theorem}
\begin{remark}
One may remark that the critical parameter $\ell_c$ in the case of dimension $d=2$ is independent of the exact law of the field we consider, under generic assumptions it will always be equal to $0$. However, for $d\geq 3$, the value of $\ell_c$ may depend on the law of the field. Nevertheless, Theorem \ref{thm:dim3} states that although we do not know the exact value of $\ell_c$ in dimension $d\geq 3$, this critical parameter $\ell_c$ is always negative, meaning that it is strictly easier to percolate in a space of dimension $d\geq 3$ than in the plane.
\end{remark}
\begin{remark}
 Note that the statements of the two theorems \ref{thm:dim2} and \ref{thm:dim3} may hold under weaker assumptions. We refer the reader to \cite{Threshold} and \cite{lcd3} to see the exact assumptions that were made on the law of $f$, as for the sake of clarity we prefer to not write the details here.
\end{remark}
When $\ell>\ell_c$ we say that we are in the \textit{supercritical phase}, when $\ell<\ell_c$ we say that we are in the \textit{subcritical phase}. This phase transition has received a lot of attention over the past few years. One result of interest is the following
\begin{theorem}[\cite{HA_critical},\cite{Threshold} for $d=2$, \cite{severo2022uniqueness} for $d\geq 3$]
\label{thm:unicity}
If $d\geq 2$ and if $q$ satisfy Assumption \ref{a:a1} for some $\beta>d$, then if $\ell >\ell_c$, the set $\excur(f)$ almost surely presents a unique unbounded component.
\end{theorem}
This theorem already gives a good description of the supercritical phase. The set $\excur(f)$ contains only one unbounded connected component and (countably) many finite ones.
In this article we aim to obtain a more geometric information about the unbounded component of $\excur(f)$.
In order to do so we introduce the following notations and definitions.
\begin{definition}
Let $\ell\in \R$ and $A,B$ be two subsets of $\R^d$, we define the event $\left\{A\connects B\right\}$, also denoted by $\left\{A \overset{f\geq -\ell}{\longleftrightarrow}B\right\}$, as the event that there exists a connected component of $\excur(f)$ that intersects both $A$ and $B$.
When $A=\{x\}$ and $B=\{y\}$ are two singletons we simply use the notation $\left\{x\connects y\right\}$ instead of $\left\{\{x\}\connects \{y\}\right\}$.
\end{definition}
\begin{remark}
 We also use the notation $\{A\overset{f\leq \ell}{\longleftrightarrow}B\}$ as a shorthand for $\{A \overset{\mathcal{E}_{\ell}(-f)}{\longleftrightarrow}B\}.$ And we similarly define $\{x\overset{f\leq \ell}{\longleftrightarrow}y\}.$
\end{remark}
We then define the chemical distance as a generalisation of the graph distance for the case of Bernoulli percolation.
\begin{definition}
\label{def:chemical_distance_random}
Given $x,y\in \R^d$ and a subset $\mathcal{E}\subset \R^d$ we define the set $\Gamma(\mathcal{E},x,y)$ as
\begin{equation}
    \Gamma(\mathcal{E},x,y) := \{\gamma : [0,1] \to \mathcal{E}\ |\ \gamma \text{ is continuous and rectifiable}, \gamma(0)=x,\gamma(1)=y\}.
\end{equation}
The set $\Gamma(\mathcal{E},x,y)$ can be empty, for instance if $x$ or $y$ are not in $\mathcal{E}$ or are not in the same connected component of $\mathcal{E}$.
The chemical distance between $x$ and $y$ is denoted by $d_\chem^{\mathcal{E}}(x,y)$, it takes values in $\mathbb{R}_+\cup \{\infty\}$ and is defined as
\begin{equation}
    \label{eq:def_dchem}
    d_\chem^{\mathcal{E}}(x,y) := \inf\{\text{length}(\gamma)\ |\ \gamma\in \Gamma(\excur(f),x,y)\},
\end{equation}
with the convention: $\inf \emptyset = \infty$.
Here, $\text{length}(\gamma)$ denotes the Euclidean length of the curve $\gamma$ ($\gamma$ being rectifiable).
\end{definition}
\begin{remark}
 One may suggest the case where we have two points $x$ and $y$ in the same connected component of $\excur(f)$, but with no rectifiable path between them, hence having $x\connects y$ but $\dchem(x,y)=\infty$. We will see that we do not need to concern ourselves with such a possibility since an analysis of our arguments shows that this event has zero probability.
\end{remark}
Our main theorem is the following:
\begin{theorem}
\label{thm:theoreme_principal}
Assume $q$ satisfies Assumption \ref{a:a1} with some $\beta>d$ and that $\ell>-\ell_c(q)$. Then for any $\delta>0$, there exists a constant $\kappa'>0$ such that
\begin{equation}  
    \Proba{0\connects x\text{ and } \dchem(0,x) > \norm{x}{}\log^{\kappa(\delta)}(\norm{x}{})} = O\left(\norm{x}{}^{-\kappa'}\right),
\end{equation}
where $\kappa(\delta) = (1+\delta)(d-1)\left(\frac{1}{2}+\frac{1}{2\beta-d}\right).$
\end{theorem}
We make a few comments about Theorem \ref{thm:theoreme_principal}. The first observation is that this theorem is to be compared with the following result of Peter Antal and Agoston Pisztora concerning the chemical distance in the context of Bernoulli percolation.
\begin{theorem}[\cite{Antal}]
\label{thm:antal}
In the context of Bernoulli percolation, if $d\geq 2$ and $p>p_c(d)$, we have constants $c,\rho>0$ (depending on $d$ and $p$) such that
$$\Proba{0 \overset{\mathcal{G}_p(\omega)}{\longleftrightarrow} x\text{ and } \text{d}_\text{chem}^{\mathcal{G}_p(\omega)}(0,x)> \rho \norm{x}{}} = O\left(e^{-c\norm{x}{}}\right).$$
In the above, $\text{d}_\text{chem}^{\mathcal{G}_p(\omega)}$ denotes the graph distance induced in $\mathcal{G}_p(\omega)$.
\end{theorem}
A major difference between the two theorems is that ours is not optimal, in the sense that it only holds for $\ell>-\ell_c$ instead of the expected full supercritical region $\ell>\ell_c$. This is due to the lack of estimates on the probability of some local uniqueness event (see Definition \ref{def:local_uniqueness_event} for the precise definition of this event). This event is a non-monotonous event and we only manage to prove that it has high probability in the regime $\ell>-\ell_c$. Note that in \cite{vernotte2023chemical}, Theorem \ref{thm:theoreme_principal} was proven for the case of dimension $d=2$ for the full supercritical region $\ell>0$. This is due to the fact that in dimension $2$ one can use planarity arguments and work only with increasing events. 
However, if indeed this local uniqueness event has high probability in the whole region $\ell>\ell_c$, we expect our proof of Theorem \ref{thm:theoreme_principal} to hold in this region (at the very least our argument would work in the region $\ell>0$). Another difference between the two theorems is the logarithm factor which appears in Theorem \ref{thm:theoreme_principal}. This is mainly due to the fact that contrary to Bernoulli percolation where the status of all sites are independent, our field $f$ has values that are heavily 
correlated. This is problematic when we try to develop a renormalization argument, since we need to consider boxes of a logarithmic scale instead of boxes of fixed size. Finally, a last difference is the speed of decrease in the two theorems. We obtain a polynomial decay in Theorem \ref{thm:theoreme_principal} instead of an exponential decay in Theorem \ref{thm:antal}. The reason for this is the fact that a continuous field $f$ can have level sets that contort a lot. While for Bernoulli percolation there is a minimal scale, this is absolutely not the case for a continuous field. Hence, we need to do a local control around the two points we want to connect, $0$ and $x$, to ensure that the chemical distance "generated" near those points is not too high. This is done via a quantitative implicit function theorem. However the estimates obtained are not strong enough to guarantee an exponential decay. This is the only place in the proof where we lose a super-polynomial decay.

We also mention that the problem of understanding chemical distance in the context of percolation was addressed for other models. In \cite{Ding2018} a result is obtained for the Gaussian free field on $\mathbb{Z}^2$ as well as for critical random walk loop soups. In \cite{peretz2025chemicaldistancelevelsets} a result similar to the one of Peter Antal and Agoston Pisztora is obtained for the level sets of the Gaussian free field on $\Z^d$ with $d\geq 3$. In  \cite{DRS14} a general theorem was obtained to obtain a result similar to Theorem \ref{thm:antal} for fields defined on $\mathbb{Z}^d$ that satisfy some weak conditions on the decay of correlations. The main novelty in the present work is the fact that there is no underlying lattice. This makes the control of the chemical distance near one point a new difficulty that was not present in previous settings.

Compared to previous work in dimension 2 \cite{vernotte2023chemical} where planarity was heavily used, the arguments used here are much closer to the ones in \cite{Antal}. The strategy of the proof can be summarized as follows. In Section \ref{sec:2} we prove that in the regime $\ell>-\ell_c$ a local uniqueness event (see Definition \ref{def:local_uniqueness_event}) holds with high probability for $f$, but also for a discrete and finite range approximation of the field $f$. In Section \ref{sec:3}, we prepare a renormalization argument. The content of this section is pretty classical but was formulated to adapt to our framework. The purpose of the renormalization is to build a path of adjacent boxes in which the field $f$ (rather its finite range approximation) satisfies the local uniqueness event. This will allow us to find a path in $\excur(f)$ of controllable length, starting not too far from $0$ and arriving not too far from $x$. Section \ref{sec:4} is dedicated to proving that one can connect $0$ and $x$ to this path by paying a reasonable cost (in terms of chemical distance). This part of the paper is inspired by the arguments developed in \cite{BG16}. We conclude Section \ref{sec:4} with the proof of Theorem \ref{thm:theoreme_principal}.

\textbf{Acknowledgements: }I am very grateful to my PhD advisor Damien Gayet for introducing me to this problem as well as for many insightful discussions and for his remarks on a preliminary version of this paper. I would also like to thank the referee for their comments and remarks on a first version of this paper.

\section{On the local uniqueness event}
\label{sec:2}
In this section we present the local uniqueness event and show that it has high probability in the region $\ell>-\ell_c$.

We begin by properly defining the local uniqueness event as the intersection of two events, one is an "existence" event while the other one is a "uniqueness" event (the later being the hard and interesting part). First we introduce some notations.
For $R>0$ we define the box $B_R\subset \R^d$ as
\begin{equation}
    \label{eq:def_BR}
    B_R := [-R,R]^d.
\end{equation}
For $i\in \{1,\dots,d\}$ we define the two $i$-faces $F_R^{i,+}$, $F_R^{i,-}$ of $B_R$ as
\begin{align*}
    F_R^{i,-} & := [-R,R]^{i-1}\times \{-R\}\times [-R,R]^{d-i}, \\
    F_R^{i,+} & := [-R,R]^{i-1}\times \{+R\}\times [-R,R]^{d-i}.
\end{align*}
\begin{definition}
\label{def:existence_event}
For $R>1$ and $\ell \in \mathbb{R}$, we define the existence event $\text{Exist}(R,\ell)$ as the event that for all $i\in \{1,\dots,d\}$, there exists a connected component of $B_R\cap\excur(f)$ that intersects both $F_R^{i,-}$ and $F_R^{i,+}$ (see Figure \ref{fig:illustration_exist} for an illustration of this event).
\end{definition}
\begin{figure}
    \centering
    \includegraphics[width=8cm]{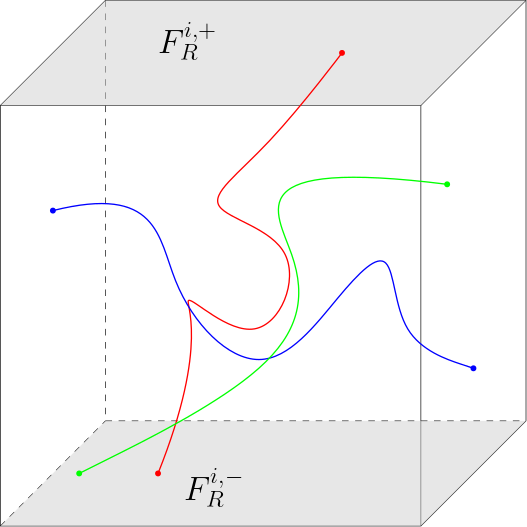}
    \caption{Illustration of the event $\text{Exist}(R,\ell)$ for $d=3$. We have three different crossings (represented in different colors), each joining two opposite faces of the box.}
    \label{fig:illustration_exist}
\end{figure}
We also define the uniqueness event.
\begin{definition}
\label{def:unique_event}
For $R>1$, $\ell \in \mathbb{R}$, we can consider the set $B_R \cap \excur(f)$ as a union of disjoint connected components $B_R \cap \excur(f) = \bigsqcup_{i\in I}\mathcal{C}_i$. Given some $\kappa \in ]0,1[$, the uniqueness event $\text{Unique}(R,\ell, \kappa)$ is the event that all $\mathcal{C}_i$ of diameter at least $\kappa R$ are connected within $B_{R(1+\kappa)} \cap \excur(f)$.
\end{definition}
\begin{remark}
 One could define the uniqueness event as the event that at most one of the $\mathcal{C}_i$ in the above definition has diameter greater than $\kappa R$, however, in order to avoid many technicalities later on, we prefer to work with this weaker definition.
\end{remark}
The local uniqueness event in the box $B_R$ is simply the intersection of the existence event and the uniqueness event in this box, more formally we make the following definition.
\begin{definition}
\label{def:local_uniqueness_event}
For $R>1$, $\ell \in \mathbb{R}$ and $\kappa\in ]0,1[$ we define the local uniqueness event $\mathcal{A}(R,\ell,\kappa)$ as
\begin{equation}
    \label{eq:local_uniqueness_event}
    \mathcal{A}(R,\ell,\kappa) := \text{Exist}(R,\ell) \cap \text{Unique}(R,\ell, \kappa).
\end{equation}
\end{definition}
We make a few comments about this local uniqueness event. These comments are not useful for the proof of our theorem but give some context about this type of event and its importance. First, note that having a high probability of the event $\mathcal{A}(R,\ell,\kappa)$ can be seen as a quantitative statement of Theorem \ref{thm:unicity}. Indeed the proof of Theorem \ref{thm:unicity} uses non quantitative ideas coming from ergodicity and the Cameron-Martin formula and does not quite describe the geometry of the unbounded cluster. Second, recall that an event is said to be \textit{increasing} if increasing values of the field $f$ can only favor the realisation of the event (that is $E$ is an increasing event if for all $f,g$ we have $(f\leq g \text{ and }f\in E)\Rightarrow g\in E$). The existence event in Definition \ref{def:existence_event} is clearly an increasing event, as the set $\excur(f)$ can only grow when $f$ increases. However, the uniqueness event in Definition \ref{def:unique_event} is neither increasing nor decreasing. Thus, our event $\mathcal{A}_R$ itself is neither increasing nor decreasing. Hence, the whole machinery designed to deal with monotonous event fails to work with the event $\mathcal{A}(R,\ell, \kappa)$, so that it is not trivial to obtain good estimates about the probability of this event. If we consider Bernoulli percolation it is proved (see for instance \cite{Pisztora1996}) that the event $\mathcal{A}(R,\ell, \kappa)$ has high probability in the whole supercritical phase (it is trivial that this event can not have high probability in the subcritical phase). It is conjectured (see for instance \cite{lcd3}) that the same is true in the context of continuous Gaussian field. In dimension $d\geq 3$ little is known about the probability of this event. In \cite{lcd3} a weaker result concerning some form of a cluster "present everywhere" was proved, but did not guarantee high probability for the event $\mathcal{A}(R,\ell, \kappa)$ even for $\ell>0$. In \cite{strong_unicity_1} and \cite{strong_unicity_2} analogues of the event $\mathcal{A}(R,\ell,\kappa)$ were proven to have high probability in the whole the region $\ell>\ell_c$ but for different percolation models. In our paper we prove that the event $\mathcal{A}(R,\ell, \kappa)$ holds with high probability as soon as $\ell>-\ell_c$ for a large class of Gaussian fields. By Theorem \ref{thm:dim3}, this result is weaker than the expected region $\ell>\ell_c$, but it still gives some  non trivial region in which the result holds. 

This section is dedicated to the proof of the following statement.
\begin{proposition}
\label{prop:strong_perco1}
Assume that $q$ satisfies assumption \ref{a:a1} for some $\beta>d$. Suppose $\ell>-\ell_c$, $\kappa\in ]0,1[$, then there exist constants $c,C>0$, such that, for all $R>1$,
\begin{equation}
    \label{eq:local_uniqueness_for_f}
    \Proba{f\in \mathcal{A}(R,\ell,\kappa)} \geq 1 -CR^de^{-cR}. 
\end{equation}
\end{proposition}
The strategy to prove Proposition \ref{prop:strong_perco1} is to argue that when $\ell>-\ell_c$ then the set $\{f\leq -\ell\}$ does not percolates (all components are bounded), and to use a deterministic argument that states that it is not possible to observe two huge connected components of $\{f\geq -\ell\}$ without having also a huge connected component in $\{f\leq -\ell\}$.

\subsection{Approximation of the field}
We introduce approximations of the field $f$ that have good properties (finite-range-dependence, piece-wise constant). This relies on the white noise decomposition of the field $f$.

In the following of the paper, we fix an arbitrary function $\chi : \R^d \to [0,1]$ such that
\begin{itemize}
    \item $\chi$ is $\mathcal{C}^\infty$. 
    \item $\chi(x)=1$ for all $x$ such that $\norm{x}{}\leq \frac{1}{4}$.
    \item $\chi(x)=0$ for all $x$ such that $\norm{x}{}\geq \frac{1}{2}.$
\end{itemize}
With this function fixed we can make the following definition.
\begin{definition}
\label{def:fR}
For $r>1$ we define $\chi_r(x) := \chi(x/r)$ and we define the random field $f_r$ as
\begin{equation}
    \label{eq:def_fr}
    f_r := (q\chi_r)\ast W.
\end{equation}
\end{definition}
Since $q\chi_r$ has compact support included in the Euclidean ball of radius $r/2$ we see that the field $f_r$ is \textit{$r$-dependent}, meaning that two events that depend on values of the field $f_r$ at distance at least $r$ are independent.
\begin{definition}
\label{def:discretization}
For $\varepsilon>0$, and a function $g: \R^d \to \R$, we define the \textit{$\varepsilon$-discretization} of $g$ as the function $g^\varepsilon : \R^d \to \R$ defined by
$g^\varepsilon(x) := g(y)$, where $y$ is the unique point in $\varepsilon\Z^d$ such that $x\in y+\left[-\frac{\varepsilon}{2},\frac{\varepsilon}{2}\right[^d.$ 
\end{definition}
We will often write $f_r^\varepsilon$ which should be interpreted as $(f_r)^\varepsilon,$ that is the $\varepsilon$-discretization of the field $f_r$.
It is good to have in mind that the fields $f_r$ and $f_r^\varepsilon$ are good local approximations of the field $f$ as $r$ is big and $\varepsilon$ is small. This can be seen in the following classical proposition which is an application of the Borell-TIS inequality.

\begin{proposition}[see for instance {\cite[Proposition 2.1]{Severo}} ]
\label{prop:severo}
Assume that $q$ satisfies Assumption \ref{a:a1} for some $\beta>d$. There exists a constant $c>0$ such that
\begin{align}
\forall r>1,\quad \forall s>\frac{1}{r^{\beta-\frac{d}{2}}}, \quad \Proba{\sup_{x\in B_1}|f_r(x)-f(x)|\geq s}\leq e^{-cs^2r^{2\beta-d}}. \\
\forall r>1,\quad\forall \varepsilon>0,\quad \forall s>\varepsilon, \quad \Proba{\sup_{x\in B_1}|f_r^\varepsilon(x) -f_r(x)| \geq s} \leq e^{-cs^2/\varepsilon^2}.
\end{align}
Where we recall that $B_1$ denotes $B_1=[-1,1]^d$.
\end{proposition}

\subsection{Deterministic argument for local uniqueness}
We now present a deterministic argument that provides a criterion to ensure that a set has the local uniqueness property.

\begin{definition}
\label{def:cluster_property}
Let $R>1, \kappa\in ]0,1[$ and recall that $B_R = [-R,R]^d$. Let $E\subset \R^d$ be a subset. We make the following definitions.
\begin{itemize}
    \item $E$ has the \textit{$(R,\kappa)$ small clusters property}, if all the connected components of $E\cap B_{R(1+\kappa)}$ have diameter less than $\kappa R.$
    \item $E$ has the \textit{$(R,\kappa)$ unique cluster property}, if all the connected components of $E\cap B_R$ of diameter greater than $\kappa R$ are connected within $E\cap B_{R(1+\kappa)}$.
    \item $E$ has the $R$ \textit{crossing property}, if for all $i\in \{1,d\}$ there exists a connected component of $E\cap B_R$ that intersects the two faces $F_R^{i,-}$ and $F_R^{i,+}$.
\end{itemize}
In this context the diameter refer to the usual Euclidean diameter.
\end{definition}

We argue that if a set has some kind of small cluster property then its complement in the box has some form of a unique cluster property. The proof we provide next is purely geometrical and deterministic (although we will soon apply it to random sets).
\begin{proposition}
\label{prop:determinitic}
Let $\kappa\in \left]0,1\right[$, $E^-\subset \R^d$ and $E^+\subset \R^d$ be two $\mathcal{C}^1$-differentiable submanifolds of dimension $d$ with boundary (these manifolds don't need to be connected and can have countably many connected components). Suppose also that their boundary is the same $E^0 = \partial E^- = \partial E^+$ and that we can write $\R^d = (E^-\setminus E^0)\sqcup E^0 \sqcup (E^+\setminus E^0)$. Then, if for some $R>1$, $E^0$ only intersects $\partial B_R$ transversely and if $E^-$ has the $(R,\kappa)$ small clusters property, then $E^+$ has the $(R,\kappa)$ unique cluster property and the $R$ crossing property.
\end{proposition}
\begin{proof}
Since the proof is technical we first provide the intuition behind our reasoning, in hopes that it will facilitate the reading of the proof. We also recommend having Figure \ref{fig:deterministic} in mind while reading the proof. The proof essentially relies on two claims. First, let a point $x\in E^+$. We claim that if the connected component of $x$ in $E^+$ has diameter greater or equal to $\kappa R$, then there is no connected component of $E^-$ that "surrounds" this connected component of $x$ (this is due do the fact that any connected component of $E^-$ has small diameter). Second, let us consider two such points $x,y\in E^+$ whose connected components in $E^+$ are of diameter greater or equal than $\kappa R$ and let us try to connect $x$ and $y$ in $E^+$ by a straight segment. This can fail since the segment may encounter connected components of $E^-$. However the first idea ensures that whenever we encounter such an obstacle (a connected component of $E^-$), we are actually coming from the "exterior" of this obstacle hence we can follow along the boundary of this obstacle until we can continue our travel along the segment from $x$ to $y$. Moreover the detours we made by traveling along the boundaries of the obstacles cannot be too big since the obstacles itself are not too big. In the following we will try to formalize such ideas, in particular the notion of interior and exterior can be made rigorous thanks to Brouwer's seperation theorem (see \cite{guillemin2010differential} for instance).

We can write $E^- \cap B_{R(1+\kappa)} = \sqcup_{i=1}^n \mathcal{C}_i$ where the $\mathcal{C}_i$ are connected. Since $E^-$ has the $(R,\kappa)$ small clusters property, we have by definition that for all $i$, $\text{diam}(\mathcal{C}_i) \leq \kappa R$. For each $i$ we can consider the topological boundary of $\mathcal{C}_i$, that is we can write $\partial \mathcal{C}_i = \sqcup_{j=1}^{n_i} S_j^{(i)}$ where the $S_j^{(i)}$ are connected compact hypersurfaces. Because of our hypothesis, all these hypersurfaces are subsets of $E^0\cup \partial B_{R(1+\kappa)}$. The Brouwer's separation theorem (\cite{guillemin2010differential}) allows us to state that each of these compact hypersurfaces $S$ separates the space $\R^d$ into two connected parts, the interior $\text{Int}(S)$ that is bounded and the exterior of $S$ that is unbounded $\text{Ext}(S)$. Without restriction of generality since the $\mathcal{C}_i$ are connected, we can assume that for all $2\leq j \leq n_i$ we have $S_j^{(i)}\subset \text{Int}(S_1^{(i)}).$ This allows us to write $\mathcal{C}_i=\text{Int}(S_1^{(i)})\setminus \bigcup_{j=2}^{n_i} \text{Int}(S_j^{(i)}).$ Denote by $\mathcal{C}$ the set of points $x\in E^+\cap B_R$ such that the connected component of $x$ in $E^+\cap B_R$ is of diameter greater than $\kappa R$. Our goal is to prove that any two points of $\mathcal{C}$ are connected within $E^+ \cap B_{R(1+\delta)}$. For this purpose, consider two points $x$ and $y$ in $\mathcal{C}$. Since $x$ is connected to a point $z$ within $B_{R(1+\kappa)}\cap E^+$ at distance at least $\kappa R$ from $x$, we deduce that $x$ does not belong to any $\text{Int}(S^{(i)}_1)$ (otherwise this would contradict the fact that all $\mathcal{C}_i$ have diameter less than $\kappa R$). By symmetry the same is true for $y$. Now, consider $\gamma$ the segment that joins the two points $x$ and $y$ (the segment stays in $B_R$ by convexity of $B_R$, however it may leave $E^+$). Follow the segment from $x$ to $y$ until we either reach $y$ or reach a point $z_1$ in some $\mathcal{C}_i$. In the later case the previous discussion allow us to state that the point $z_1$ belongs to some $S^{(i)}_1$. We can then look at $\tilde{z_1}$ the last intersection point between $\gamma$ and $S^{(i)}_1$ and replace the segment $[z_1,\tilde{z_1}]$ of the path $\gamma$ by some path in $S^{(i)}_1$ that connects the two points $z_1$ and $\tilde{z_1}$ (see Figure \ref{fig:deterministic} for an illustration). Since $z_1$ is at distance at least $\kappa R$ from the boundary  $\partial B_{R(1+\kappa)}$ and since $\mathcal{C}_i$ is of diameter at most $\kappa R$ we see that $\mathcal{C}_i$ does not intersects $\partial B_{R(1+\kappa)}$. Thus, the outer boundary $S_1^{(i)}$ is included in $E^0$. Using this observation, we see that we can indeed find a path between $z_1$ and $\tilde{z_1}$ in $E^0$ (and also in $B_{R(1+\kappa)}$). By repeating this procedure until we reach $y$, we have connected the two points $x$ and $y$ within $E^+\cap B_{R(1+\kappa)}$, proving that $E^+$ has the $(R,\kappa)$ unique cluster property. Note that $\mathcal{C}$ can be shown to be non empty simply by choosing two points in $B_R\setminus \bigcup_{i=1}^n \text{Int}(S^{(i)}_1)$ at distance $\kappa R$ from one each other and connecting them as done above. This concludes the proof that $E^+$ has he $(R,\kappa)$ unique cluster property. In order to show that $E^+$ has the $R$ crossing property, it is enough for each direction $j\in \{1,\dots,d\}$ to find two points $w^j_1,w^j_2$ such that $w_1^j\in F_R^{j,-}\cap \mathcal{C}$ and $w_2\in F_R^{j,+}\cap \mathcal{C}$, as they will be connected in $B_R\cap E^+$ via $\mathcal{C}$. In order to algorithmically build such the point $w_i^1$, one way is to draw a segment starting from one point in $\mathcal{C}$ towards the center of the face $F_R^{j,-}$. Should the segment encounter one of $\mathcal{C}_i$ that does not intersects $F_R^{j,-}$, we deviate the path by following the boundary of $\mathcal{C}_i$ until we can continue on our initial trajectory by staying in $E^+$. If the segment reaches some $\mathcal{C}_i$ that intersects $F_R^{j,-}$, we deviate our path by following the boundary of $\mathcal{C}_i$ until we reach a point in $F_R^{j,-}$. Otherwise, our segment will simply reach a point in $F_R^{i,-}$. The path we built starts from a point in $\mathcal{C}$ and only visits points in $E^+$. This guarantees that all points visited by the path are in $\mathcal{C}$. Moreover the path eventually visits one point $w_1^j$ in $F_R^{j,-}$. This concludes the proof of the proposition.
\begin{figure}[h]
    \centering
    \includegraphics[width=10cm]{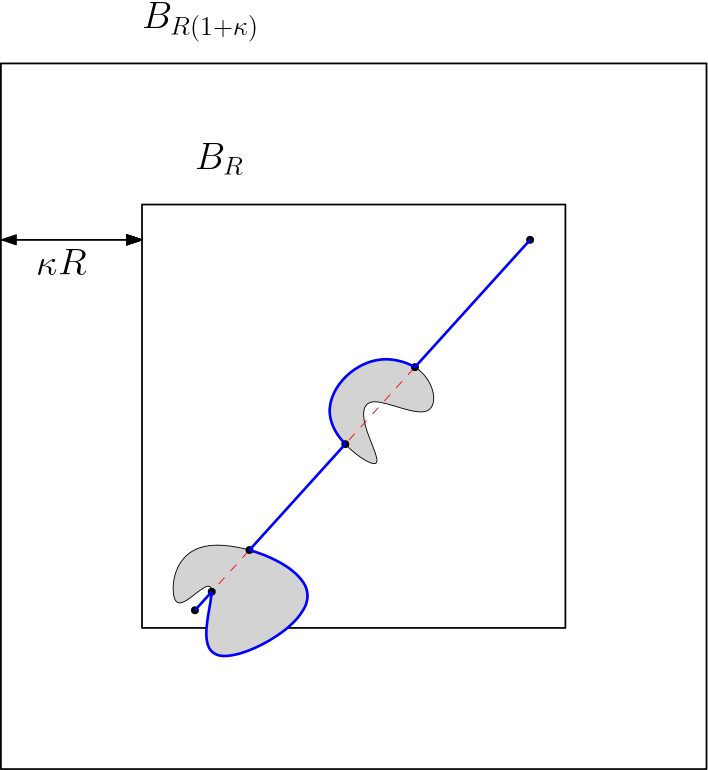}
    \caption{Illustration of the proof of Proposition \ref{prop:determinitic}.}
    \label{fig:deterministic}
\end{figure}
\end{proof}

We conclude by a remark on the definition of the unique cluster property. One may find more natural to define the following uniqueness property:
\begin{definition}
    Say that $E\subset \R^d$ has the \textit{$(R,\kappa)$ strong unique cluster property} if among the connected components of $E\cap B_R$ at most one has diameter greater than $\kappa R$.
\end{definition}
We also believe that an analogue of Proposition \ref{prop:determinitic} with the $(R,\delta)$ unique cluster property replaced by the $(R,\delta)$ strong unique cluster property should hold. However a lot of technicalities are generated due to the possible intersections between the boundary of the box $B_R$ with the boundary of the clusters of $E^-$. In order to avoid dealing these difficulties we prefer to work with the definition of the $(R,\kappa)$ unique cluster property which is enough for our purpose.

\subsection{Proof of Proposition \ref{prop:strong_perco1}}
We now provide a proof of Proposition \ref{prop:strong_perco1}. In fact we will prove the more general result that also concerns the approximation $f_r^\varepsilon$ of the field $f$.
\begin{proposition}
\label{prop:strong_perco}
Assume that $q$ satisfies Assumption \ref{a:a1} for some $\beta>d$.
Let $\ell>-\ell_c$, $\kappa\in ]0,1[$. Then, there exist constants $c,C>0$ such that for all $R>1$, $r>1$, $\varepsilon<1$,
\begin{align}
       \label{eq:local_uniqueness_for_f}
    \Proba{f\in \mathcal{A}(R,\ell,\kappa)} & \geq 1 -CR^de^{-cR}, \\
    \label{eq:local_uniqueness_for_fR}
    \Proba{f_r^\varepsilon\in \mathcal{A}(R,\ell,\kappa)} &\geq 1 -CR^d(e^{-cR}+e^{-cr^{2\beta-d}}+e^{-c\varepsilon^{-2}}).
\end{align}
\end{proposition}
The strategy is to show that the complements of $\excur(f)$ and $\mathcal{E}_\ell(f_r^\varepsilon)$ have the $(R,\kappa)$ small cluster property (see Definition \ref{def:cluster_property}) in order to deduce, via Proposition \ref{prop:determinitic}, that these sets have the $(R,\kappa)$ unique cluster property.
First we recall an estimate about a crossing probability for Gaussian field.
\begin{theorem}[\cite{HA_critical}, \cite{Threshold} for $d=2$, \cite{Severo} for $d\geq 3$]
\label{thm:severo_crossing_prob}
Assume that $q$ satisfies Assumption \ref{a:a1} for some $\beta>d$.
For all $\ell<\ell_c$, there exist constants $c,C>0$ such that for all $R\geq 1,$
\begin{equation}
    \Proba{B_1 \overset{f\geq -\ell}{\longleftrightarrow} \partial B_R}\leq Ce^{-cR}.
\end{equation}
\end{theorem}
We will use this theorem in the form of the equivalent corollary (which is just a direct consequence of the equality in law of $f$ and $-f$).
\begin{corollary}
\label{cor:severo_crossing_proba}
Assume that $q$ satisfies Assumption \ref{a:a1} for some $\beta>d$.
Then, for $\ell>-\ell_c$, there exist constant $c,C>0$ such that for all $R\geq 1,$
\begin{equation}
    \Proba{B_1 \overset{f\leq -\ell}{\longleftrightarrow} \partial B_R}\leq Ce^{-cR}.
\end{equation}
\end{corollary}
\begin{proof}
Since $f$ and $-f$ have the same law, because $f$ is a centered Gaussian field, then the two sets
$$\excur(f) =\{x\in \R^d\  |\ f(x)\geq -\ell\} \text{ and}$$
$$\excur(-f)=\{x\in \R^d\ |\ -f(x)\geq -\ell\}$$ have the same law. The later set can be written as $\{x\in \R^d | f(x) \leq -(-\ell)\}$. Since taking $\ell<\ell_c$ is equivalent to taking $-\ell>-\ell_c$, the corollary holds.
\end{proof}
We use this estimate to show that, with high probability, the box $B_R=[-R,R]$ does not contain a path of $\{f\leq -\ell\}$ with high diameter when $\ell>-\ell_c$.

\begin{proposition}
\label{prop:high_prob_no_big_diameter}
Assume that $q$ satisfies Assumption \ref{a:a1} for some $\beta>d$.
Then, for all $\ell>-\ell_c$ and $\kappa\in ]0,1[$, there exist constants $c,C>0$ such that, for all $R>1$, $r>1$,  $\varepsilon<1$,
\begin{align}
    \label{eq:high_prob_no_big_diameter}
      & \Proba{
     \left\{x\in \R^d\ |\  f_r^\varepsilon(x)\leq -\ell\right\}\text{ has  the } (R,\kappa) \text{ small clusters property}
     }
     \nonumber\\
    \geq \quad &1-CR^d(e^{-cR}+e^{-c\varepsilon^{-2}}+e^{-cr^{2\beta-d}}).
\end{align}
Moreover, if we replace $f_r^\varepsilon$ by $f$ we get the following:
\begin{align}
    \label{eq:high_prob_no_big_diameter_f}
     &\Proba{\{x\in \R^d\ |\  f(x)\leq -\ell\}\text{ has  the } (R,\kappa) \text{ small clusters property}} \nonumber\\
    \geq \quad & 1-CR^de^{-cR}.
\end{align}
\end{proposition}
\begin{proof}
Fix $\ell'\in \R$ such that $\ell>\ell'>-\ell_c$.
Cover the box $B_{R(1+\kappa)} = [-(1+\kappa)R,(1+\kappa)R]^d$ with $N\leq CR^d$ boxes of size $1$ (where $C$ is a constant depending only on the dimension $d$) and denote $b_1,b_2,\dots,b_N\in \R^d$ the centers of all those boxes. On the event that $\{f\leq -\ell'\}$ does not have the $(R,\kappa)$ small clusters property, we can find $x$ and $y$ in $B_{R(1+\kappa)}$ such that $x$ and $y$ are distant of at least $\kappa R$ and such that $x$ and $y$ both are in a same connected component of $\{f\leq -\ell'\}$. Let $i$ be one index such that $x\in b_i+[-1,1]^d$. As soon as $R$ is big enough (compared to $\frac{1}{\kappa}$ and $d$), we have $y\not\in b_i+\left[-\frac{\kappa R}{2\sqrt{d}},\frac{\kappa R}{2\sqrt{d}}\right]^d$. Hence on the event that $\{f\leq -\ell'\}$ does not have the $(R,\kappa)$ small clusters property, we have
$$\exists 1\leq i \leq N,\quad b_i+[-1,1]^d\overset{f\leq -\ell'}{\longleftrightarrow}\partial\left(b_i+\left[-\frac{\kappa R}{2\sqrt{d}},\frac{\kappa R}{2\sqrt{d}}\right]^d\right).$$
Applying stationarity (to assume $b_i=0$) and Corollary \ref{cor:severo_crossing_proba}, we can do an union bound on all possible $i$ to get
\begin{align}
    & \Proba{\{f\leq -\ell'\} \text{ does not have the }(R,\kappa)\text{ small clusters property}}\nonumber\\
    \leq\quad &\sum_{i=1}^N \Proba{b_i+[-1,1]^d \overset{f\leq -\ell'}{\longleftrightarrow}\partial\left(b_i+\left[-\frac{\kappa R}{2\sqrt{d}},\frac{\kappa R}{2\sqrt{d}}\right]^d\right)} \nonumber\\
    \leq\quad& C'R^de^{-c'R},
\end{align}
where $c'$ and $C'$ are constants that only depend on $\ell$, $\kappa$ and $d$.
This yields the conclusion \eqref{eq:high_prob_no_big_diameter_f} for the field $f$. For the field $f_r^\varepsilon$, we know by local comparison (see Proposition \ref{prop:severo}) that for $r$ big enough and $\varepsilon$ small enough (compared to $|\ell-\ell'|$), we have a constant $c>0$ depending on the law of $f$ and on $|\ell-\ell'|$ such that
\begin{equation}
    \Proba{\sup_{B_1}|f-f_r^\varepsilon|>|\ell-\ell'|}\leq e^{-cr^{2\beta-d}}+ e^{-c\varepsilon^{-2}}.
\end{equation}
Hence, using again an union bound we see that
\begin{equation}
    \Proba{\sup_{B_{R(1+\kappa)}}|f-f_r^\varepsilon|>|\ell-\ell'|}\leq C R^d(e^{-cr^{2\beta-d}}+ e^{-c\varepsilon^{-2}}).
\end{equation}
Now, if $\sup_{B_{R(1+\kappa)}}|f-f_r^\varepsilon|<|\ell-\ell'|$ then for all $x\in B_{R(1+\kappa)}$ we see that
$$f_r^\varepsilon(x)\leq -\ell \Rightarrow f(x)\leq -\ell'.$$
Thus, if $\sup_{B_{R(1+\kappa)}}|f-f_r^\varepsilon|<|\ell-\ell'|$ and if $\{x\in \R^d\ |\ f(x)\leq -\ell'\}$ has the $(R,\kappa)$ small cluster property, then so does $\{x\in \R^d\ |\ f_r^\varepsilon(x)\leq -\ell\}$. By an union bound we conclude:
\begin{align*}
& \Proba{\{x\in \R^d\ |\  f_r^\varepsilon(x)\leq -\ell\}\text{ does not have  the } (R,\kappa) \text{ small clusters property}} \\
\leq\quad & \Proba{\{x\in \R^d\ |\  f(x)\leq -\ell'\}\text{ does not have  the } (R,\kappa) \text{ small clusters property}} \\
& + \Proba{\sup_{B_{R(1+\kappa)}}|f-f_r^\varepsilon|>|\ell-\ell'|} \\
\leq\quad& C R^d(e^{-cr^{2\beta-d}}+ e^{-c\varepsilon^{-2}})+CR^de^{-cR}.
\end{align*}
Adjusting constants we get \eqref{eq:high_prob_no_big_diameter}.
\end{proof}

Finally, we are ready to prove that when $\ell>-\ell_c$ the local uniqueness event $\mathcal{A}(R,\ell,\kappa)$ has high probability.
\begin{proof}[Proof of Proposition \ref{prop:strong_perco}]
We will do the proof for $f$, the proof is straightforward to adapt to $f_r^\varepsilon$.
Denote $$E^+ := \{x\in \mathbb{R}\ |\ f(x)\geq -\ell\},$$
$$E^- := \{x\in \R^d\ |\ f(x)\leq -\ell\},$$
$$E^0 := \{x\in \R^d\ |\ f(x)=-\ell\}.$$
We also fix $R>1$.
Applying Lemma A.9 of \cite{quasi_independance}, we see that almost surely, the sets $E^+, E^-, E^0$ satisfy the hypotheses of Proposition \ref{prop:determinitic} for $B_R$. Hence, whenever $E^-$ has the $(R,\kappa)$ small clusters property, then $E^+$ will have the $(R,\kappa)$ unique cluster property and $R$ crossing property. It only remains to apply our estimates of Proposition \ref{prop:high_prob_no_big_diameter}, to see that both events $\{f\in \text{Unique}(R,\ell,\kappa)\}$ and $\{f\in \text{Exist}(R,\ell)\}$ have probability no less than $1 -CR^de^{-cR}.$
\end{proof}
\begin{remark}
The way we proved that the existence event $\text{Exist}(R,\ell)$ has high probability may seem unsatisfactory. We argue that it is possible to also derive this result very easily for all $\ell>0$ simply by considering the percolation model induced in two-dimensional planes and use the results in \cite{Threshold} to show that in the plane rectangle are crossed with high probability. Also, since the event $f\in \text{Exist}(R,\ell)$ is an increasing connection event, the machinery developed by Severo in \cite{Severo} together with Theorem 3.1 in \cite{Pisztora1996} might allow to prove this result for all $\ell>\ell_c$, but a lot remains to be done in this direction.
\end{remark}

\section{Stochastic domination argument}
\label{sec:3}
In this section we introduce several tools for a renormalization argument. The idea is that we want to build a path of adjacent boxes into a renormalized lattice where in each box, the field $f_r^\varepsilon$ will satisfy the local uniqueness event (see Definition \ref{def:local_uniqueness_event}). In order to do so, we first recall a classical stochastic domination theorem. This theorem allows us to compare a probability measure $\mu$ on $\{0,1\}^{\Z^d}$ with \textit{high marginals} and \textit{finite range correlation} with a classical Bernoulli process of high parameter. By \text{high marginals}, we mean that if $(\omega_i)\in \{0,1\}^{\Z^d}$ is sampled according to $\mu$, then for all $i\in \Z^d$ the probability of $\omega_i$ being $1$ is close to $1$ (uniformly in $i$). Typically one can have in mind the law of a configuration $\omega\in \{0,1\}^{\Z^d}$ where $\omega_i=1$ if and only if the field $f_r^\varepsilon$ satisfy the local uniqueness event in the box $\frac{r}{10}i+B_r$. Then we introduce the notion of \textit{global structure} as some path in the renormalized lattice that is of controllable length and comes close to two selected points. We argue that such structures appear with high probability. This last part relies on arguments developed in \cite{Newman} and \cite{Antal}. 

\subsection{Comparison with a Bernoulli process of high parameter}

We briefly recall the notion of stochastic domination for two measures on $\{0,1\}^{\Z^d}$ and then state a classical stochastic domination result.


First we make the following definition.
\begin{definition}
Consider a countable set $V$, and two probability measures $\mu_1$ and $\mu_2$ on $\left(\{0,1\}^V, \mathcal{B}(\{0,1\})^{\otimes V}\right)$. We say that $\mu_1$ is stochastically dominated by $\mu_2$ and we note $\mu_1 \prec \mu_2$ if there exist a coupling $(X^1,X^2)$ of law $\mu$ where $X^1$ has law $\mu_1$ and $X^2$ has law $\mu_2$ and where $$\mu\left(\forall i\in V,\   X^1_i \leq X^2_i\right)=1.$$
\end{definition}

\begin{definition}
Let $d\geq 2$ and $p\in [0,1]$, denote by $\pi_p$ the law on $\{0,1\}^{\Z^d}$ given by independent and identically distributed Bernoulli random variables of parameter $p$.
We immediately have $$p\leq p'\Rightarrow \pi_p \prec \pi_{p'}.$$
\end{definition}
We also recall the notion of measures that are finite-range dependent.
\begin{definition} Let $M\in \mathbb{N}$. A probability measure $\mu$ on $\{0,1\}^{\Z^d}$ is said to be \textit{$M$-dependent}, if for all subsets $A,B$ of $\Z^d$ such that $d_\infty(A,B)\geq M$ the two $\sigma$-algebras $\mathcal{F}_A:= \sigma(\omega_i, i\in A)$ and $\mathcal{F}_B := \sigma(\omega_i, i\in B)$ are $\mu$-independent. Here and in the following $d_\infty$ denotes the metric induced by the sup-norm, meaning that $d_\infty(A,B)=\min_{a,b\in A\times B}\max_{1\leq i \leq d}{|a_i-b_i|}.$
\end{definition}
A classical and general tool to compare a probability measure $\mu$ with some $\pi_p$ is the following theorem. Basically the idea is that a measure $\mu$ with finite-range dependence stochastically dominates a Bernoulli process $\pi_{1-\alpha}$ with small $\alpha>0$ as soon as the marginals of $\mu$ are close enough to $1$.
\begin{theorem}[\cite{LSS}, \cite{Antal}]
\label{prop:stoch:p2}
For any $M\in \mathbb{N}$, $d\geq 1$ and $p\in [0,1]$, there exist a quantity $\alpha(p,M,d)\in [0,1]$ such that 
\begin{equation}
    \label{eq:non_triviality}
    \alpha(p,M,d) \xrightarrow[p\to 1]{}0,
\end{equation}
and such that for any law $\mu$ on $\{0,1\}^{\Z^d}$ satisfying
\begin{itemize}
    \item $\mu$ is $M$-dependent (finite range dependence),
    \item $\inf_{i\in \Z^d}\mu(\omega_i=1)\geq p$ (high marginals),
\end{itemize}
we have
\begin{equation}
    \pi_{1-\alpha(p,d,M)}\prec \mu.
\end{equation}
\end{theorem}
\begin{remark}
Note that the fact that $\alpha(p,M,d)$ can be as small as we wish as soon as $p$ get close to $1$ is the core statement of the claim, otherwise we could just take $\alpha(p,M,d)\equiv 1$ since we always have $\pi_0 \prec \mu$, but of course this is not what we are interested in. Furthermore, an analysis of the proofs in \cite{Antal} says that we can always choose $\alpha(p,d,M)\leq 4(1-p)^\frac{1}{(2M+1)^d}$.
\end{remark}

\subsection{Global structures}
We introduce some terminology related to percolation on discrete graphs and we also define the notion of \textit{global structure} that will be used in the proof of Theorem \ref{thm:theoreme_principal}.

\begin{definition}
    Consider $A\subset \Z^d$.
    We say that two sites $x,y\in A$ are connected (resp. $\star$-connected) within $A$ if one can find a finite sequence $x_0,\dots,x_n$ of points of $A$ such that
    \begin{itemize}
        \item $x_0 = x$, $x_n = y$,
        \item $\forall 0\leq i \leq n-1,\ \norm{x_{i+1}-x_{i}}{1}=1$ (resp. $\norm{x_{i+1}-x_{i}}{\infty}=1$).
    \end{itemize}
    We say that $A$ is connected (resp. $\star$-connected) if any two points in $A$ are connected (resp $\star$-connected) within $A$.
\end{definition}
\begin{remark}
 Note that one should not confound the notion of being connected (as a subset of $\Z^d$) with the notion of being connected (as a subset of $\R^d$). In the rest of this section, the above definition takes precedence.
\end{remark}
In the following, let $\omega\in \{0,1\}^{\Z^d}$ be a configuration.
We now state the definition of a \textit{global structure} around two points.
\begin{definition}
\label{def:global_structure}
    Let $x\in \Z^d\setminus \{0\}$ and $C_0,\delta>0$ be constants.
    A set $G\subset \Z^d$ is said to be a \textit{global structure} around $0$ and $x$ with constants $C_0,\delta$ if the following is satisfied
    \begin{itemize}
        \item $\forall y \in G,\ \omega_y=1$,
        \item $|G|\leq C_0\norm{x}{1},$
        \item $G$ is connected,
        \item $d(0,G)\leq \log^{1+\delta}(\norm{x}{1})$, and
        \item $d(x,G)\leq \log^{1+\delta}(\norm{x}{1})$.
    \end{itemize}
    We also define $\mathcal{G}(x,C_0,\delta)$ the event that there exists a global structure around $0$ and $x$ with constants $C_0,\delta$.
\end{definition}
In this section we prove the following result
\begin{proposition}
\label{prop:newman}
For all $d\geq 2$, there exists $0<\alpha_d<1$ such that the following holds. For all $\delta>0$ there exist constants $c,C,C_0>0$ such that, if $p\in [1-\alpha_d,1]$, we have
\begin{equation}
    \forall x\in \Z^d \setminus\{0\},\ \pi_p\left(\mathcal{G}(x,C_0,\delta)\right)>1-Ce^{-c\log^{1+\delta}(\norm{x}{1})}.
\end{equation}
\end{proposition}
\begin{proof}
We follow the argument in \cite{Pisztora1996}. Consider a deterministic path of length $N= \norm{x}{1}$ in $\mathbb{Z}^d$ joining $0$ and $x$. We denote this path by $x_0,x_1,\dots,x_N$ with $x_0=0$, $x_N=x$ and $\norm{x_{i+1}-x_{i}}{1}=1$ for all $0\leq i \leq N-1$.
For each $0\leq i \leq N$ we can consider the biggest $\star$-connected component of $\{j\in \Z^d\ |\ \omega_j=0\}$ that contains $i$ (that is, the union of all sets that simultaneously are $\star$-connected, contain $i$ and only contain vertices $j$ such that $\omega_j=0$). Denote this $\star$-connected component by $\mathcal{C}_i$ (note that if $\omega_i=1$ then $\mathcal{C}_i = \emptyset$). For a set $\mathcal{C}\subset \Z^d$ we define the \textit{boundary} of $\mathcal{C}$ as $$\partial \mathcal{C} := \{i\in \Z^d\ |\ i\not\in \mathcal{C}\text{ and }\exists j\in \mathcal{C}, \norm{i-j}{\infty}=1\}.$$
We define $\overline{\mathcal{C}_i}$ as
$$\overline{\mathcal{C}_i}:=\begin{cases}
    \{i\} & \text{ if }\mathcal{C}_i=\emptyset \\
   \mathcal{C}_i \cup \partial{\mathcal{C}_i} &\text{otherwise.}
\end{cases}$$
We also define
$$A:=\bigcup_{0\leq i \leq N}\overline{\mathcal{C}_i}.$$
Using the arguments developed in \cite{Newman} (see also section 4 in \cite{Antal}), it is known that there exists $\alpha_d>0$ such that for $1>p>1-\alpha_d$ we have constants $C_0,c>0$ such that
\begin{equation}
    \pi_p(|A|>C_0 N)\leq \exp(-cN).
\end{equation}
We further claim that with high probability, $A$ is contained in a $\log^{1+\delta}(\norm{x}{1})$ neighborhood of the path $x_0,x_1,\dots,x_N$.
More precisely, a classical Peierls argument shows that, if $p>1-\frac{1}{3^d}$, there exists a constant $c>0$ such that for all $n\geq 0$
\begin{equation}
    \label{eq:peierls}
    \pi_p\left(|\mathcal{C}_0|>n\right)\leq \exp(-cn).
\end{equation}
Denote by $\mathcal{N}$ the event that $A$ is contained in a $\log^{1+\delta}(\norm{x}{1})$ neighborhood of the path $x_0,\dots,x_N$, that is $\mathcal{N}$ is the event $$\{\forall a \in A\ ,\  \exists 0\leq i \leq N\ ,\  \norm{a-x_i}{1}\leq \log^{1+\delta}(\norm{x}{1})\}.$$
Thus choosing $n$ of order $\log^{1+\delta}(\norm{x}{1})$ in \eqref{eq:peierls} and doing an union bound for all clusters $\mathcal{C}_i$ for $0\leq i \leq N$, we obtain (adjusting constants)
\begin{equation}
    \pi_p(\mathcal{N})\leq Ce^{-c\log^{1+\delta}(\norm{x}{1})}.
\end{equation}
Finally notice that if $G$ denotes the external boundary of $A$, then $G$ is a set of open sites which is connected and that surrounds the segment $x_0,\dots,x_N$, hence the conclusion.
\end{proof}

\section{Local control of the chemical distance and proof of the main theorem}
\label{sec:4}
In this section we propose an argument to control the probability of having an unusual high chemical distance in a box. We will then use these estimates together with the result of the previous section to give the proof of Theorem \ref{thm:theoreme_principal}.

\subsection{Local control of the chemical distance around a point}

We introduce some notations that were defined in \cite{vernotte2023chemical}. The definitions were made for the two dimensional case, but they straightforwardly adapt to a higher dimension setting.

First, we recall Definition \ref{def:chemical_distance_random} of the chemical distance and we make the following definition.
\begin{definition}
\label{def:chem_diameter}
Let $\mathcal{E}$ and $\mathcal{C}$ be two subsets of $\R^d$ 
The \textit{chemical diameter} of $\mathcal{C}$ within $\mathcal{E}$ is defined as
\begin{equation}
    \label{eq:def_chem_diameter}
    \diam^\mathcal{E}_\chem(\mathcal{C}):= \sup\{d^\mathcal{E}_\chem(x,y)\ |\ x,y\in \mathcal{C}\}.
\end{equation}
\end{definition}
Since we will be interested in a set that may have several connected components we make the following definition.
\begin{definition}
    \label{def:SR}
    Let $s>1$ and $\mathcal{E}\subset \R^d$ be a subset. We assume that $\mathcal{E}\cap B_s$ has finitely many connected components denoted by $\mathcal{C}_1\dots, \mathcal{C}_n$.
    We define
    \begin{equation}
        \label{eq:defSR}
        S(s,\mathcal{E}):=\max_{1\leq i \leq n} \diam^\mathcal{E}_\chem(\mathcal{C}_i).
    \end{equation}
    In the case where $\mathcal{E}\cap B_s$ has infinitely many connected components, we arbitrarily define $S(s,\mathcal{E})=\infty.$
\end{definition}
Our claim is the following:
\begin{proposition}
\label{prop:local} Assume that $q$ satisfies Assumption \ref{a:a1} for some $\beta>d$.
Then, for all $\ell\in \R$, for all $\eta>0$, there exist constants $C,c,\kappa, \kappa'>0$ such that
\begin{equation}
    \forall s>10,\  \forall x>10,\  \Proba{S(s,\excur(f))\geq x}\leq C\frac{s^\kappa}{x^{\kappa'}}+ e^{-c s^\eta}.
\end{equation}
\end{proposition}
The estimate in Proposition \ref{prop:local} is pretty weak, we believe the same result to hold with a much faster decay but we could not find a proof of such a result. However, this estimate will only be used two times in two small boxes of logarithmic scale.
In order to prove Proposition \ref{prop:local}, we will show that the box $B_s$ can be filled by small boxes of size $\varepsilon$ such that, with high probability, for each such small box $B$, the level set $\{f=-\ell\}$ either avoids $B$ or cuts it into only two parts. This property can be interpreted in some sense as the level set $\{f=-\ell\}$ not being too much tortuous. Our argument is inspired by the work in \cite{BG16}. We state a few results from \cite{BG16} that we will use in the proof of Proposition \ref{prop:local}. First, note that if $A\subset \R^d$ is a set and $f\in \mathcal{C}^2(\R^d,\R)$, we use the Euclidean nature of $\R^d$ to confound the first derivative $df$ with the gradient $\nabla f$ and the second derivative $d^2 f$ with the Hessian matrix $H(f)$. We also define $\norm{d f}{\mathcal{C}^1(A)} := \sup_{x\in A}\norm{\nabla f(x)}{}+\sup_{x\in A}\norm{H(f)(x)}{}.$

\begin{proposition}[Quantitative implicit function theorem]
\label{prop:implicit}
Let $f : \R^d\to \R$ be a $\mathcal{C}^2$ function. Take $U\subset \R^d$ an open set and $x\in U$, and positive constants $k>\lambda >0$. Assume that the following is satisfied
\begin{itemize}
    \item $f(x)=0,$
    \item $x+[-1,1]^d\subset U,$
    \item $\norm{\nabla f(x)}{}>\lambda,$
    \item $\norm{df}{\mathcal{C}^1(U)}\leq k.$
\end{itemize}
Then, if $\varepsilon = \frac{\lambda^2}{4k^2d^{3/2}}$ there exists a $\mathcal{C}^2$ function $\phi : \mathbb{R}^{d-1} \to \mathbb{R}$, and a direction $i\in \{1,\dots,d\}$ such that
$$\forall y\in x+[-\varepsilon,\varepsilon]^d, f(y)=0 \Leftrightarrow y_i = \phi(y_1,\dots,\widehat{y_i},\dots,y_d).$$
\end{proposition}
\begin{proof}
Since we have $\norm{\nabla f(x)}{}>\lambda$ we can find a $i\in \{1,\dots,d\}$ such that
$$\left|\frac{\partial f}{\partial x_i}(x)\right|\geq \frac{\lambda}{\sqrt{d}}.$$
Fix $\delta= \frac{\lambda}{2dk}$ (note that we have $\delta \leq 1$ by hypothesis), by the mean value theorem we have for all $y\in x+[-\delta,\delta]^d$
$$\left|\frac{\partial f}{\partial x_i}(x)-\frac{\partial f}{\partial x_i}(y)\right|\leq \norm{x-y}{}\norm{df}{\mathcal{C}^1(U)}\leq \sqrt{d}\delta k=\frac{\lambda}{2\sqrt{d}}.$$
Dividing by $\left|\frac{\partial f}{\partial x_i}(x)\right|$ we obtain
$$\forall y \in x+[-\delta,\delta]^d, \left|1-\left(\frac{\partial f}{\partial x_i}(x)\right)^{-1}\frac{\partial f}{\partial x_i}(y)\right|\leq \frac{1}{2}.$$
We can apply Theorem A.2 in \cite{BG16} with a $M\leq \frac{\sqrt{d}}{\lambda}$ and $C\leq k$ which directly implies our claim for
$\varepsilon = \frac{\delta}{2MC}\geq \frac{\lambda^2}{4d^{3/2}k^2}$.
\end{proof}
We now state some well-known estimates about the behaviour of the derivative of the random field $f$ in a box.

\begin{proposition}
\label{prop:no_degeneracy}
Assume that $q$ satisfies Assumptions \ref{a:a1} for some $\beta>d$ and let $\ell \in \R$.
Then, there exist constants $c, C, \kappa_1, \kappa_2>0$ such that for all $u\geq 0$ and $s\geq 10$,
\begin{equation}
    \label{eq:control_norm}
    \Proba{\norm{f}{\mathcal{C}^2(B_s)}\geq C\log^{\frac{1}{2}}(s)+u} \leq \exp\left(-c\frac{u^2}{2}\right),
\end{equation}
\begin{equation}
    \label{eq:no_degeneracy}
    \Proba{\exists x \in B_s\  |\  f(x)=-\ell \text{ and } \norm{\nabla f(x)}{}<\lambda} \leq C\lambda^{\frac{1}{\kappa_1}} s^{\kappa_2}.
\end{equation}
\end{proposition}
\begin{proof}
Concerning \eqref{eq:control_norm}, this is an application of the Borell-TIS inequality, see also Lemma 5.2 in \cite{BG16} for a similar estimate. For the proof of \eqref{eq:no_degeneracy}, this is a restatement of Lemma 5.3 in \cite{BG16}.
\end{proof}
Since the proof of Proposition \ref{prop:local} requires covering various sets with a controllable number of small boxes of size $\varepsilon$ we introduce the following definition.
\begin{definition}
    Let $A\subset \R^d$ and $\varepsilon>0$. A set $\mathcal{R}\subset A$ is called a \textit{$\varepsilon$-covering} of $A$ if
    $$A\subset \bigcup_{x\in \mathcal{R}}B(x,\varepsilon),$$
    where $B(x,\varepsilon)$ denotes the open box $B(x,\varepsilon) := x + ]-\varepsilon,\varepsilon[^d$.
    We also define 
    \begin{equation}
        N(A,\varepsilon):= \inf\{|\mathcal{R}|\ |\ \mathcal{R} \text{ is a }\varepsilon\text{-covering of }A\}.
    \end{equation}
    Note that $N(A,\varepsilon)$ possibly can be infinite (when $A$ is unbounded).
\end{definition}
\begin{lemma}
\label{lemma:covering}
If $A\subset B\subset \R^d$ and $\varepsilon>0$ then
$$N(A,2\varepsilon)\leq N(B,\varepsilon).$$
\end{lemma}
\begin{proof}
The statement is trivial if $N(B,\varepsilon)=\infty$, otherwise choose $\mathcal{R}=\{x_1,\dots,x_n\}\subset B$ a $\varepsilon$-covering of $B$. We have $A\subset B \subset \bigcup_{i=1}^n B(x_i, \varepsilon)$. Find $0\leq m\leq n$ and reorder the $x_i$'s so that $A\subset \bigcup_{i=1}^m B(x_i,\varepsilon)$ and so that $A$ intersects all $B(x_i, \varepsilon)$ for $1\leq i \leq m$. For $1\leq i \leq m$ we choose $y_i \in A\cap B(x_i,\varepsilon)$ and since $B(x_i,\varepsilon)\subset B(y_i, 2\varepsilon)$ we immediately see that $y_1,\dots,y_m$ is a $2\varepsilon$-covering of $A$, which yields the conclusion.
\end{proof}
\begin{corollary}
\label{cor:covering}
There exists a constant $C$ depending only on $d$ such that for all $A\subset B_R$ and $\varepsilon>0$ we have
$$N(A,\varepsilon)\leq C\left(\frac{R}{\varepsilon}\right)^d.$$
\end{corollary}
\begin{proof}
We apply Lemma \ref{lemma:covering} to $B=B_R$ and we note that $B_R$ can obviously be covered with $O\left((\frac{R}{\varepsilon})^d\right)$ balls of the form $B(x,\frac{\varepsilon}{2}).$
\end{proof}
We now provide the proof of Proposition \ref{prop:local}.
\begin{proof}[Proof of Proposition \ref{prop:local}.]
Fix $s\geq 10$, $x\geq 10$ and $\eta>0$.
Let $\lambda>0$  to be fixed later. Denote
$$k:=s^\eta.$$
Consider the two events:
$$\mathcal{A} := \left\{\norm{df}{\mathcal{C}^1(B_s+[-1,1]^d)}\leq k\right\},$$
$$\mathcal{B} := \left\{\forall y \in B_s,\  f(y)=-\ell \Rightarrow \norm{\nabla f(x)}{}>\lambda\right\}.$$
Applying Proposition \ref{prop:no_degeneracy}, we see that
\begin{equation}
    \label{eq:int1}
\Proba{\mathcal{A}}\geq 1-e^{-cs^{2\eta}} \quad \text{and}\quad\Proba{\mathcal{B}}\geq 1-C\lambda^{\frac{1}{\kappa_1}}s^{\kappa_2}.
\end{equation}
Set $\varepsilon = \frac{\lambda^2}{4d^{3/2}k^2}$ as in Proposition \ref{prop:implicit}.
Applying Corollary \ref{cor:covering} to the set $B_s\cap \{f=-\ell\}$ with $\frac{\varepsilon}{2}$, we can find a constant $C_d$ (depending only on $d$) and a collection of $N\leq C_d\left(\frac{s}{\varepsilon}\right)^d$ points denoted by $(x_i)_{1\leq i \leq N}$ such that for all $i$, $f(x_i)=-\ell$ and such that
$$\forall y\in B_s,\  f(y)=-\ell\Rightarrow \exists 1\leq i \leq N,\ y\in B(x_i, \frac{\varepsilon}{2}).$$
Moreover, by the triangular inequality, if $y\in B_s$ is such that $d(y,\{f=-\ell\})\leq \frac{\varepsilon}{2}$, then there exists $1\leq i \leq N$ such that $y\in B(x_i, \varepsilon)$.

Consider the set $F := \{y\in B_s\ |\ \exists z\in B_s, f(z)=-\ell\text{ and }d(y,z)\leq \frac{\varepsilon}{2}\}$ and its complement $G := B_s \setminus F$. The previous observation shows that one has
$$F\subset \bigcup_{i=1}^N B(x_i,\varepsilon).$$
Applying Corollary \ref{cor:covering}, we can also cover $G$ by less than $N$ balls of radius $\frac{\varepsilon}{2\sqrt{d}}$ and of centers $y_i$ chosen in $G$, that is
$$G \subset \bigcup_{i=1}^N B\left(y_i,\frac{\varepsilon}{2\sqrt{d}}\right).$$
Note that by construction, for each ball $B\left(y_i, \frac{\varepsilon}{2\sqrt{d}}\right)$ in the covering of $G$ this ball does not intersect $\{f=-\ell\}\cap B_s$.

Under the event $\mathcal{A}\cap \mathcal{B}$ take any two points $z_1,z_2\in B_s\cap \mathcal{E}_\ell(f)$.
We claim the following.
\begin{claim}
\label{claim:claim_cd}
There exists a deterministic constant $C_d'>0$ depending only on $d$, such that under one of the following hypotheses.
\begin{itemize}
    \item There exists $i\in \{1,\dots,N\}$ such that $z_1$ and $z_2$ are in $\overline{B(x_i, \varepsilon)}$,
    \item There exists $i\in \{1,\dots,N\}$ such that $z_1$ and $z_2$ are in $\overline{B\left(y_i, \frac{\varepsilon}{2\sqrt{d}}\right)}$ and $z_1$ is in $B\left(y_i, \frac{\varepsilon}{2\sqrt{d}}\right)$,
\end{itemize}
then there exists a rectifiable path in $\mathcal{E}_\ell(f)$ of length less than $C'_d\varepsilon$ that connects $z_1$ and $z_2$.
\end{claim}
\begin{proof}[Proof of Claim \ref{claim:claim_cd}]
Assume first that $z_1$ and $z_2$ are in some $\overline{B(x_i,\varepsilon)} = x_i+[-\varepsilon,\varepsilon]^d$. We apply Proposition \ref{prop:implicit} at $x_i$ and we see that the box $\overline{B(x_i,\varepsilon)}$ is separated in two by the hypersurface $\{f=-\ell\}$. Since by hypothesis we have $z_1,z_2\in \mathcal{E}_\ell(f)$ we deduce that the two points $z_1$ and $z_2$ are on the same side of this hypersurface. These two points can connect within $\mathcal{E}_\ell(f)$ first by traveling towards one face of the box and then travelling along this face of the box to connect. This connection was made of length less than $4\varepsilon+2\sqrt{d}\varepsilon.$

Now assume that $z_1$ and $z_2$ are in some $\overline{B\left(y_i, \frac{\varepsilon}{2\sqrt{d}}\right)}$. By a previous observation, the set $B\left(y_i,\frac{\varepsilon}{2\sqrt{d}}\right)\cap B_s$ does not intersects $\{f=-\ell\}$, thus by continuity of $f$ this set $B\left(y_i, \frac{\varepsilon}{2\sqrt{d}}\right)\cap B_s$ is either included in $\{f>-\ell\}$ or included in $\{f<-\ell\}$. By the assumption on $z_1$ the later case can be excluded. By taking the closure we see that $\overline{B\left(y_i, \frac{\varepsilon}{2\sqrt{d}}\right)}\subset \mathcal{E}_\ell(f)$. Now by convexity, the result follows, as we can simply join the two points $z_1$ and $z_2$ by a straight segment.
\end{proof}
We can now conclude the proof of Proposition \ref{prop:local}. Consider two points $y$ and $z$ that are connected in $B_s\cap \mathcal{E}_\ell(f)$, and denote by $\gamma : [0,1] \to B_s\cap \excur(f)$ a continuous function such that $\gamma(0)=y$ and $\gamma(1)=z$. We want to find a path between $y$ and $z$ of controlled length. To do so consider the following construction. We start from $z_0 := y$, if $z_0\in F$ we know that we have some $i$ such that $z_0 \in B(x_i, \varepsilon)$ and we define $\mathcal{B}_0 := B(x_i, \varepsilon)$ (for such an arbitrarily chosen $i$). Otherwise there exists some $i$ such that $z_0 \in B\left(y_i, \frac{\varepsilon}{2\sqrt{d}}\right)$ and we define $\mathcal{B}_0 := B\left(y_i, \frac{\varepsilon}{2\sqrt{d}}\right)$ for such a $i$ (again arbitrarily chosen). We then define $t_0 := \sup\{t>0\ |\ \gamma(t)\in \mathcal{B}_0\}$ and  $z_1=\gamma(t_0)$ and we repeat this procedure until we reach step $M$ where $t_{M-1}=1$. We have obtained a sequence $(z_k)_{0\leq k \leq M}$ (with $z_0=y$ and $z_M=z$) and boxes $(\mathcal{B}_k)_{0\leq k \leq M-1}$. Moreover, by construction, all boxes obtained in this construction are different (this also guarantees the fact that the procedure ends). Since there are at most $2N$ boxes to choose from, we have $M\leq 2N$. Furthermore, for all $k$ we have $z_k \in \mathcal{B}_k$ and $z_{k+1}\in \overline{\mathcal{B}_k}$. Additionally, we have $z_k,z_{k+1}\in B_s \cap \mathcal{E}_\ell(f)$ (this is due to the fact that the $z_k$ are part of the path $\gamma$).
We can thus apply Claim \ref{claim:claim_cd} to find a path in $\mathcal{E}_\ell(f)$ joining $z_k$ and $z_{k+1}$ of length less than $C'_d\varepsilon$. By gluing together these paths, we obtain a rectifiable path between $y$ and $z$ of length less than $2NC'_d\varepsilon.$ Since the points $y$ and $z$ were arbitrary, on the event $\mathcal{A}\cap \mathcal{B}$ we have $S(s,\mathcal{E}_\ell(f)) \leq 2NC'_d\varepsilon \leq 2C_dC'_d\frac{s^d}{\varepsilon^{d-1}}.$
By our choice of $\varepsilon$ we see that we have a constant $C''_d$ depending on $d$ such that
$$S(s,\mathcal{E}_\ell(f))\leq C''_d s^d(s^\eta)^{2(d-1)}\lambda^{-2(d-1)}.$$
Finally, we choose $$\lambda = \left(\frac{C''_ds^{d+2\eta(d-1)}}{x}\right)^{\frac{1}{2(d-1)}},$$ which yields that under the event $\mathcal{A}\cap \mathcal{B}$ we have $S(s,\excur(f))\leq x$.
Using estimates \eqref{eq:int1} and adjusting constants, we get the conclusion of Proposition \ref{prop:local}.
\end{proof}

\subsection{Proof of Theorem \ref{thm:theoreme_principal}}
We now give the proof of Theorem \ref{thm:theoreme_principal}.
\begin{proof}[Proof of Theorem \ref{thm:theoreme_principal}.]
In this proof $c,C$ denote positive constants that may change from line to line. By isotropy, it is enough to look at the chemical distance between $0\in \R^d$ and $(x,0,\dots,0)\in \R^d$ where $x>1$ is a real number (that can be thought of as big). Thus in the following, we consider some $x$ bigger than some $x_0>1$. The value of $x_0$ will be determined later in the proof. Fix also some parameter $\delta>0$, that is free for now. We introduce the following quantities
$$r=R=\log^{\frac{1+\delta}{2\beta-d}}(x) \ ,\ \ \varepsilon = \frac{1}{\log^{\frac{1+\delta}{2}}(x)}.$$
Also, consider, $\ell$ and $\ell'$ such that $$\ell>\ell'>-\ell_c.$$
We will work at the level $\ell'$ with $f_r^\varepsilon$ to then recover information about the field $f$ at level $\ell$.

We introduce our renormalization procedure in order to use Theorem \ref{prop:stoch:p2}.
Define the event $\mathcal{A}_0$ as
$$\mathcal{A}_0 := \left\{f_r^\varepsilon \in \mathcal{A}\left(R,\ell',\frac{1}{100}\right)\right\},$$
where we recall Definition \ref{def:local_uniqueness_event} of the event $\mathcal{A}(R,\ell,\kappa)$. Note that we decide to work with $\kappa=\frac{1}{100}$ simply to fix a certain value for $\kappa$, but the exact value is not crucial for the rest of the argument.
For $i\in \mathbb{Z}^d$ we define the event $\mathcal{A}_i$ as the translation of the event $\mathcal{A}_0$ of vector $\frac{R}{10}i$. We also denote by $\mathbf{B}_i$ the box $\mathbf{B}_i = \frac{R}{10}i+B_R$ and by $\tilde{\mathbf{B}}_i$ the box $\tilde{\mathbf{B}}_i=\frac{R}{10}i+B_{R(1+\frac{1}{100})}$. With these notations, note that the event $\mathcal{A}_i$ can simply be restated as the occurrence of the existence event and the uniqueness event in the box $\mathbf{B}_i$.

Consider the configuration $(\omega_i)_{i\in \Z^d}\in \{0,1\}^{\Z^d}$ such that $\forall i\in \Z^d, \ \omega_i =\mathds{1}_{\mathcal{A}_i}$, and denote by $\mu$ the probability measure on $\{0,1\}^{\Z^d}$ whose law is the law of $\omega \in \{0,1\}^{\Z^d}.$
By construction, using the fact that $f_r^\varepsilon$ is $r$-dependent, we see that $\mu$ is $100$-dependent (where $100$ is clearly not optimal but will not intervene later on). We can apply Theorem \ref{prop:stoch:p2} and find $\alpha(x, \delta)\in [0,1]$ such that
$$\pi_{1-\alpha(x,\delta)}\prec \mu.$$
Furthermore, by stationarity of the law of the field, $\mu(\omega_i=1)$ does not depend on $i$. Thus, $$\inf_{i\in \Z^d}\mu(\omega_i=1)=\mu(\omega_0=1)= \Proba{f_r^\varepsilon\in \mathcal{A}\left(R,\ell',\frac{1}{100}\right)}.$$
By Proposition \ref{prop:strong_perco} we thus obtain: 
\begin{align*}
    \inf_{i\in\Z^d}\mu(\omega_i=1) &\geq\quad  1-CR^d\left(e^{-cR}+e^{-cr^{2\beta-d}}+e^{-c\varepsilon^{-2}}\right) \\
     & \geq\quad  1-C\left(\log^{\frac{1+\delta}{2\beta-d}}(x)\right)^d\left(e^{-c\log^{\frac{1+\delta}{2\beta-d}}(x)}+e^{-c\log^{1+\delta}(x)}+e^{-c\log^{1+\delta}(x)}\right).
\end{align*}
Where in the last line we used the explicit definitions of $R$ and $r$ in terms of $x$ and $\delta$.
Note that letting $x$ goes to infinity this yields
$$\inf_{i\in\Z^d}\mu(\omega_i=1) \xrightarrow[x\to \infty]{}1.$$

By the conclusion of Theorem \ref{prop:stoch:p2}, we see that $\alpha(x,\delta)$ can be chosen so that
$$\alpha(x,\delta) \xrightarrow[x\to \infty]{}0.$$
We can thus take $x_0>1$ (depending on $d,\ell,\delta,q$) big enough so that for all $x>x_0$ we have $1-\alpha(x,\delta)>1- \alpha_d$ where $\alpha_d$ is given by Proposition \ref{prop:newman}. In the following we assume that $x$ is bigger than $x_0$.
Let $M\in \mathbb{N}^*$ such that $x\in \frac{R}{10}M+B_R.$
Recall Definition \ref{def:global_structure} of the event $\mathcal{G}(M,C,\delta)$. By Proposition \ref{prop:newman}, we can choose $C_0>0$ depending on $\delta$ and $d$ such that
\begin{equation}
    \pi_{1-\alpha_d}\left(\mathcal{G}(M,C_0,\delta)\right) \geq 1 - Ce^{-c\log^{1+\delta}(M)},
\end{equation}
for some constants $c,C>0$.
Since, $\mu$ stochastically dominates $\pi_{1-\alpha_d}$ and since the event $\mathcal{G}(M,C_0,\delta)$ is an increasing event, we get
\begin{equation}
    \mu\left(\mathcal{G}(M,C_0,\delta)\right) \geq 1 - Ce^{-c\log^{1+\delta}(M)}.
\end{equation}
Under the realization of this event, denote by $\mathcal{C}\subset \mathbb{Z}^d$ a global structure around $(0,\dots,0)$ and $(M,0\dots,0)$ with constants $C_0,\delta$. Since for all $i\in \mathcal{C}$ the event $\mathcal{A}_i$ occurs, and using the overlap of the boxes, we can glue together the "big clusters" of $\{f_r^\varepsilon \geq -\ell'\}\cap \mathbf{B}_i$ of all boxes $\mathbf{B}_i$ for $i\in \mathcal{C}$ within the union of the $\tilde{\mathbf{B}}_i$ for $i\in \mathcal{C}$ (this is due to Definition \ref{def:unique_event} of the uniqueness event with $\kappa=\frac{1}{100}$). In particular this implies that there exists a path $\gamma$ in $\{f_r^\varepsilon\geq -\ell'\}$ that stays in the $\tilde{\mathbf{B}}_i$ for $i\in \mathcal{C}$ and such that
\begin{equation}
\label{eq:e37}
    \max(d(0,\gamma),d(x,\gamma))\leq R\log^{1+\delta}(M),
\end{equation}
 where $d$ is the Euclidean distance in $\mathbb{R}^d$. Since we have $M\leq x$ and $R=\log^{\frac{1+\delta}{2\beta-d}}(x)$ we obtain that
\begin{equation}
    \label{eq:e38}
\max(d(0,\gamma),d(x,\gamma))\leq \log^{1+\delta+\frac{1+\delta}{2\beta-d}}(x).
\end{equation}
Note that since such a path is confined in at most $C_0M$ boxes $\tilde{\mathbf{B}}_i$ and since $f_r^\varepsilon$ is piece-wise constant on each small cube of size $\varepsilon$, this path $\gamma$ can be chosen of Euclidean length less than $C'M\varepsilon\frac{R^d}{\varepsilon^d}$ where $C'$ is a constant that may depend on $d$ and $C_0$. This is a very rough upper bound that we obtain by dividing each box $\tilde{\mathbf{B}}_i$ into $C'\left(\frac{R}{\varepsilon}\right)^d$ small boxes of size $\varepsilon$ on which the field $f_r^\varepsilon$ is constant, and then by choosing $\gamma$ so that if $\gamma$ visits one of this small boxes it will never visit it again.
With the additional observation that $M\leq \frac{x}{R}$, the above can be restated as follows: under the event $\mathcal{G}(M,C_0,\delta)$ there exists a continuous rectifiable path $\gamma$ taking values in $\{f_r^\varepsilon \geq -\ell'\}$ such that \eqref{eq:e38} is verified and such that
\begin{equation}
    \text{length}(\gamma)\leq C'x\log^{(1+\delta)(d-1)(\frac{1}{2}+\frac{1}{2\beta-d})}(x).
\end{equation}
Moreover, the above construction shows that $\gamma$ can be chosen included in a box of the form $B_{C''x}$ where $C''>0$ is a constant depending on $C_0,\delta, d$.

Applying Proposition \ref{prop:severo} we see that for $x$ big enough we have a constant $c>0$ such that
\begin{equation}
    \Proba{\norm{f-f_r^\varepsilon}{B_{C''x}}>\ell-\ell'} \leq e^{-c\log^{1+\delta}(x)}.
\end{equation}
Under the event $\mathcal{H}:=\left\{\norm{f-f_r^\varepsilon}{B_{C''x}}\leq \ell-\ell'\right\}$
 we observe that the path $\gamma$ is also included in $\excur(f)$ (since on this event we have $f_r^\varepsilon(z)\geq -\ell'$ implies $f(z)\geq f'(z)-(f'(z)-f(z)) \geq -\ell$ for all $z\in B_{C''x}$).
 
On the event that $(0,\dots,0)$ and $(x,0,\dots,0)$ are connected within $\excur(f)$ we argue that with high probability we can connect $(0,\dots,0)$ and $(x,0\dots,0)$ to the path $\gamma$ in a neighborhood of $(0,\dots,0)$ and $(x,0\dots,0)$. In fact, consider the event $\mathcal{U}_1$ as
\begin{equation}
    \mathcal{U}_1 := \left\{f\in \text{Unique}\left(2\log^{1+\delta+\frac{1+\delta}{2\beta -d}}(x),\ell, \frac{1}{100}\right)\right\},
\end{equation}
and define the event $\mathcal{U}_2$ as the translation of the event $\mathcal{U}_1$ by $(x,0,\dots,0)$ (recall the definition of the uniqueness event in Definition \ref{def:unique_event}). Applying Proposition \ref{prop:strong_perco} we see that for $x$ big enough we have
\begin{equation}
    \Proba{\mathcal{U}_1}=\Proba{\mathcal{U}_2}\geq 1-Ce^{-c\log^{1+\delta}(x)}.
\end{equation}
Under the event $\{0\connects x\}$ and the event $\mathcal{U}_1$ we know that $0$ connects to the path $\gamma$ within the box centered at $0$ of radius $2\left(1+\frac{1}{100}\right)\log^{1+\delta+\frac{1+\delta}{2\beta -d}}(x)$ (a similar connection happens around $x$ via the event $\mathcal{U}_2$). Now it only remains to argue that this connection can be done with a path of reasonable Euclidean length. For this, we use the estimate of Proposition \ref{prop:local}. More precisely define the event $\mathcal{S}_1$ as 
\begin{equation}
    \mathcal{S}_1 = \left\{S\left(2(1+\frac{1}{100})\log^{1+\delta+\frac{1+\delta}{2\beta -d}}(x), \excur(f)\right) \leq x\log^{(1+\delta)(d-1)(\frac{1}{2}+\frac{1}{2\beta-d})}(x)\right\},
\end{equation}
and define the event $\mathcal{S}_2$ as the translated of the event $\mathcal{S}_1$ by $(x,0,\dots,0)$ (recall Definition \ref{def:SR} of $S(s,E)$). Applying Proposition \ref{prop:local} and using stationarity we see that there exists a constant $\kappa'>0$ such that for $x$ big enough
\begin{equation}
    \Proba{\mathcal{S}_1}=\Proba{\mathcal{S}_2}\geq 1-x^{-\kappa'}.
\end{equation}
Under the event $\mathcal{S}_1$ it is known that the path between $0$ and $\gamma$ can be chosen of Euclidean length less than $x\log^{(1+\delta)(d-1)(\frac{1}{2}+\frac{1}{2\beta-d})}(x)$ (similarly for the connection between $x$ and $\gamma$ using $\mathcal{S}_2$). Hence, gluing these paths together we finally obtain a path between $0$ and $x$ in $\excur(f)$ of length less than $3x\log^{(1+\delta)(d-1)(\frac{1}{2}+\frac{1}{2\beta-d})}(x).$ Doing an union bound and adjusting constants we finally obtain that for $x$ big enough
$$\Proba{0\connects x \text{ and }\dchem(0,x)>3x\log^{(1+\delta)(d-1)(\frac{1}{2}+\frac{1}{2\beta-d})}(x)}\leq x^{-\kappa'}.$$
This is precisely the conclusion of Theorem \ref{thm:theoreme_principal}.

\end{proof}
David, Vernotte
\\\\
Institut Fourier, UMR 5582, Laboratoire de MathématiquesUniversité Grenoble Alpes, CS 40700, 38058 Grenoble cedex 9, France
\nocite{*}
\bibliography{biblio.bib}

\providecommand{\bysame}{\leavevmode\hbox to3em{\hrulefill}\thinspace}
\providecommand{\MR}{\relax\ifhmode\unskip\space\fi MR }
\providecommand{\MRhref}[2]{%
  \href{http://www.ams.org/mathscinet-getitem?mr=#1}{#2}
}
\providecommand{\href}[2]{#2}
\begin{thebibliography}{10}

\bibitem{Antal}
Peter Antal and Agoston Pisztora, \emph{On the chemical distance for
  supercritical bernoulli percolation}, The Annals of Probability \textbf{24}
  (1996), no.~2, 1036--1048.

\bibitem{azais}
Jean-Marc Azais and Mario Wschebor, \emph{Level sets and extrema of random
  processes and fields}, Found. Comput. Math. \textbf{10} (2010), no.~4,
  481–484.

\bibitem{BG16}
Vincent Beffara and Damien Gayet, \emph{Percolation of random nodal lines},
  Publications math{\'e}matiques de l'IH{\'E}S \textbf{126} (2017), no.~1,
  131--176.

\bibitem{Ding2018}
Jian Ding and Li~Li, \emph{{Chemical Distances for Percolation of Planar
  Gaussian Free Fields and Critical Random Walk Loop Soups}}, Communications in
  Mathematical Physics \textbf{360} (2018), no.~2, 523--553.

\bibitem{DRS14}
Alexander Drewitz, Balázs Ráth, and Artëm Sapozhnikov, \emph{On chemical
  distances and shape theorems in percolation models with long-range
  correlations}, Journal of Mathematical Physics \textbf{55} (2014), no.~8,
  083307.

\bibitem{strong_unicity_1}
Hugo Duminil-Copin, Subhajit Goswami, Pierre-François Rodriguez, and Franco
  Severo, \emph{Equality of critical parameters for percolation of gaussian
  free field level sets}, Duke Mathematical Journal (2023).

\bibitem{strong_unicity_2}
Hugo Duminil-Copin, Subhajit Goswami, Pierre-François Rodriguez, Franco
  Severo, and Augusto Teixeira, \emph{A characterization of strong percolation
  via disconnection}, 2023.

\bibitem{lcd3}
Hugo Duminil-Copin, Alejandro Rivera, Pierre-François Rodriguez, and Hugo
  Vanneuville, \emph{Existence of an unbounded nodal hypersurface for smooth
  gaussian fields in dimension d > 3}, The Annals of Probability \textbf{51}
  (2023), no.~1, 228--276.

\bibitem{Newman}
Luiz Fontes and Charles~M. Newman, \emph{{First Passage Percolation for Random
  Colorings of $\mathbb{Z}^d$}}, The Annals of Applied Probability \textbf{3}
  (1993), no.~3, 746 -- 762.

\bibitem{guillemin2010differential}
Victor Guillemin and Alan Pollack, \emph{Differential topology}, vol. 370,
  American Mathematical Soc., 2010.

\bibitem{janson}
Svante Janson, \emph{Gaussian hilbert spaces}, no. 129, Cambridge university
  press, 1997.

\bibitem{Kesten}
Harry Kesten, \emph{The critical probability of bond percolation on the square
  lattice equals 1/2}, Communications in mathematical physics \textbf{74}
  (1980), no.~1, 41--59.

\bibitem{LSS}
T.~M. Liggett, R.~H. Schonmann, and A.~M. Stacey, \emph{{Domination by product
  measures}}, The Annals of Probability \textbf{25} (1997), no.~1, 71 -- 95.

\bibitem{NOFKG}
Stephen Muirhead, Alejandro Rivera, Hugo Vanneuville, and Laurin
  K{\"o}hler-Schindler, \emph{{The phase transition for planar Gaussian
  percolation models without FKG}}, The Annals of Probability \textbf{51}
  (2023), no.~5, 1785 -- 1829.

\bibitem{Threshold}
Stephen Muirhead and Hugo Vanneuville, \emph{{The sharp phase transition for
  level set percolation of smooth planar Gaussian fields}}, Annales de
  l'Institut Henri Poincaré, Probabilités et Statistiques \textbf{56} (2020),
  no.~2, 1358 -- 1390.

\bibitem{NS16}
Fedor Nazarov and Mikhail Sodin, \emph{Asymptotic laws for the spatial
  distribution and the number of connected components of zero sets of gaussian
  random functions}, Zurnal matematiceskoj fiziki, analiza, geometrii
  \textbf{12} (2015).

\bibitem{peretz2025chemicaldistancelevelsets}
Tal Peretz, \emph{{Chemical distance for the level sets of the Gaussian free
  field}}, 2025.

\bibitem{Pisztora1996}
Agoston Pisztora, \emph{Surface order large deviations for ising, potts and
  percolation models}, Probability Theory and Related Fields \textbf{104}
  (1996), no.~4, 427--466.

\bibitem{quasi_independance}
Alejandro Rivera and Hugo Vanneuville, \emph{{Quasi-independence for nodal
  lines}}, Annales de l'Institut Henri Poincaré, Probabilités et Statistiques
  \textbf{55} (2019), no.~3, 1679 -- 1711.

\bibitem{HA_critical}
Alejandro Rivera and Hugo Vanneuville, \emph{The critical threshold for
  bargmann--fock percolation}, Annales Henri Lebesgue \textbf{3} (2020),
  169--215.

\bibitem{Severo}
Franco Severo, \emph{Sharp phase transition for {Gaussian} percolation in all
  dimensions}, Annales Henri Lebesgue \textbf{5} (2022), 987--1008 (en).

\bibitem{severo2022uniqueness}
\bysame, \emph{{Uniqueness of Unbounded Component for Level Sets of Smooth
  Gaussian Fields}}, International Mathematics Research Notices \textbf{2024}
  (2023), no.~11, 9002--9009.

\bibitem{vernotte2023chemical}
David Vernotte, \emph{Chemical distance in the supercritical phase of planar
  gaussian fields}, arXiv preprint arXiv:2312.14205 (2023).

\end{thebibliography}

\end{document}